\documentclass{lms}
\usepackage[all]{xy}
\usepackage{amsmath}
\usepackage{amssymb}
\usepackage{mathrsfs}
\usepackage{pstricks, multido}
\usepackage{comment}
\usepackage{fp}

\newcommand{\newvertices}{\verticesof'}
\newcommand{\someedges}{\mathcal{E}}
\newcommand{\insof}{\mathrm{in}}
\newcommand{\intinsof}{\insof^\inter}
\newcommand{\intedges}{\edgesof^\inter}
\newcommand{\extedges}{\edgesof^\ext}
\newcommand{\extinsof}{\insof^\ext}
\newcommand{\rootv}{*}
\newcommand{\vertextwo}{\vertex'}
\newcommand{\leaf}{\lambda}
\newcommand{\tree}{T}
\newcommand{\vertex}{v}
\newcommand{\edgesof}{E}
\newcommand{\edge}{e}
\newcommand{\verticesof}{V}
\newcommand{\conlast}{\mathfrak{c}}
\newcommand{\conleaf}{\mathfrak{a}}
\newcommand{\conroot}{\mathfrak{b}}
\newcommand{\ext}{\mathrm{ext}}
\newcommand{\inter}{\mathrm{int}}
\newcommand{\treesone}[1]{\mathbb{T}(#1)}
\newcommand{\trees}[2]{\mathbb{T}_{#2}(#1)}

\newcommand{\setofsubsets}{\mathscr{S}}
\newcommand{\finiteset}{S}
\newcommand{\finitesubset}{P}

\newcommand{\mapofsets}{f}
\newcommand{\element}{s}

\newcommand{\rR}{\mathbb{R}}
\newcommand{\thecircle}{S^1}
\newcommand{\nN}{\mathbb{N}}
\newcommand{\halfplane}{\mathcal{H}}
\newcommand{\halfplanes}{\mathbb{H}}
\newcommand{\diskone}{X}
\newcommand{\disktwo}{Y}
\newcommand{\diskthree}{Z}
\newcommand{\affc}{\aff{\cC}}

\newcommand{\interval}{I}
\newcommand{\cC}{\mathbb{C}}
\newcommand{\cog}{L}

\newcommand{\dummyfunction}{g}
\newcommand{\imaginaryunit}{i}
\newcommand{\smallepsilon}{\epsilon}

\newcommand{\realvariable}{a}
\newcommand{\timez}{t}
\newcommand{\varone}{p}
\newcommand{\vartwo}{q}
\newcommand{\cent}{c}
\newcommand{\radi}{r}
\newcommand{\vara}{u}
\newcommand{\varb}{w}
\newcommand{\varc}{x}
\newcommand{\vard}{y}
\newcommand{\vare}{z}
\newcommand{\complexunit}{\vare}
\newcommand{\centertwo}{\cent_\disktwo}
\newcommand{\centerone}{\cent_\diskone}
\newcommand{\centerthree}{\cent_\diskthree}
\newcommand{\rmin}{r_{\text{inf}}}
\newcommand{\rmax}{r_{\text{sup}}}
\newcommand{\rsmall}{r_{\text{har}}}
\newcommand{\rbig}{r_{\text{ari}}}
\newcommand{\bigR}{R}
\newcommand{\ita}{\itm}
\newcommand{\itb}{\itn}
\newcommand{\itn}{n}
\newcommand{\iti}{i}
\newcommand{\itj}{j}
\newcommand{\itk}{k}

\newcommand{\itm}{m} 
\newcommand{\nest}[2]{d(#1,#2)}
\newcommand{\nesthalf}{\nest{\diskone_{\frac{1}{2}}}{\disktwo_{\frac{1}{2}}}}
\newcommand{\nestsbase}{\min\left\{\nest{\diskone_0}{\disktwo_0},\nest{\diskone_1}{\disktwo_1}\right\}}
\newcommand{\disty}{\Delta}
\newcommand{\displacement}{\rho}
\newcommand{\integerpower}{m}
\newcommand{\integerthing}{k}
\newcommand{\fixedpower}{n}
\newcommand{\lowerrational}{\alpha}
\newcommand{\middlerational}{\beta}
\newcommand{\higherrational}{\gamma}
\newcommand{\angleof}{\theta}

\newcommand{\randomangle}{\nu}

\newcommand{\longmap}{\tau}
\DeclareMathOperator{\aff}{Aff}
\newcommand{\psplone}{\mathbf{x}}
\newcommand{\psplnaught}{\mathbf{z}}
\newcommand{\pspltwo}{\mathbf{y}}
\newcommand{\psplthree}{\psplnaught}
\newcommand{\overlappedset}{\Lambda}
\newcommand{\totallyadmissiblesubset}{\Theta}
\newcommand{\marking}{{\mathrm{Mark}}}
\newcommand{\augmarking}{\widehat{\marking}}
\newcommand{\fldel}{\zeta}

\newcommand{\dmmarking}{\overline{\marking}}
\newcommand{\comp}{\odot}
\newcommand{\conf}{\mathfrak{C}}
\newcommand{\conft}{\conf^{0}}
\newcommand{\nei}{\mathscr{N}}
\newcommand{\admissibles}{\mathbb{A}}
\newcommand{\disk}{\delta}

\DeclareMathOperator{\emb}{Emb}
\newcommand{\homotopypart}{h}
\newcommand{\longproj}{\psi}
\newcommand{\homotopy}{H}
\newcommand{\unei}{U}
\newcommand{\proj}{\mathrm{proj}}

\newcommand{\spacea}{W}
\newcommand{\spaceb}{X}
\newcommand{\spacec}{Y}
\newcommand{\spaced}{Z}
\newcommand{\spacemap}{\varphi}
\newcommand{\cover}{\mathbb{O}}
\newcommand{\Unei}{\unei}
\newcommand{\retract}{\pi}
\newcommand{\subspacegroup}{K}

\newcommand{\permut}{\sigma}
\newcommand{\msph}{\mathcal{M}}
\newcommand{\sspace}{\mathcal{U}}

\newcommand{\sS}{\mathbb{S}}
\newcommand{\unitcollection}{I}
\newcommand{\mapofcollections}{\nattran}

\newcommand{\oprealization}{\mathcal{R}}
\newcommand{\interoppushout}{\replaced{\thirdop}}

\newcommand{\einf}{\mathbb{E}}
\newcommand{\firstop}{\mathcal{O}}
\newcommand{\secondop}{\mathcal{P}}
\newcommand{\thirdop}{\mathcal{Q}}
\newcommand{\wfunctor}{\mathbb{W}}
\newcommand{\operad}{\firstop}
\newcommand{\trivialoperad}{\unitcollection}
\newcommand{\cyrc}{\circ}

\newcommand{\fld}{FLD}
\newcommand{\fldone}{FLD_1}
\newcommand{\onlycircle}[1]{#1_\bigcirc}
\newcommand{\flda}{\onlycircle{FLD}}

\newcommand{\topmon}{\mathcal{A}}
\newcommand{\gpel}{\pspltwo}
\newcommand{\opnei}{\mathscr{N}}
\newcommand{\tld}{An^{\mathrm{triv}}}
\newcommand{\tdrtwo}{\flda^{\mathrm{norm}}}
\newcommand{\Mbar}{\overline{\msph}}
\newcommand{\mbarone}{\onlycircle{\mbartwo}}
\newcommand{\mbartwo}{\overline{\mathbb{M}}}
\newcommand{\td}{\mbartwo}
\newcommand{\mbarthree}{\overline{\mathfrak{m}}}
\newcommand{\mbarthreetop}{\mbarthree^{\mathrm{Top}}}
\newcommand{\equivtodm}{\Psi}
\newcommand{\secty}{\Gamma}

\DeclareMathOperator{\id}{Id}
\DeclareMathOperator*{\colim}{colim}

\newcommand{\replacement}[1]{\overline{#1}}
\newcommand{\firstobject}{{\mathcal{W}}}
\newcommand{\secondobject}{\mathcal{X}}
\newcommand{\thirdobject}{\mathcal{Y}}
\newcommand{\fourthobject}{\mathcal{Z}}
\newcommand{\nattran}{\xi}
\newcommand{\morphismone}{\mathfrak{f}}
\newcommand{\morphismtwo}{\mathfrak{g}}
\newcommand{\functora}{F}
\newcommand{\functorb}{G}
\newcommand{\category}{\mathcal{C}}
\newcommand{\diagram}{\mathcal{D}}

\newcommand{\functor}{\functora}
\newcommand{\Scube}{\mathbf{Cu}_\finiteset}

\newcommand{\replaced}[1]{\widehat{#1}}
\newcommand{\cofar}{\ar@{>->}}
\newcommand{\cubecol}[3]{\displaystyle\colim_{\mathbf{Cu}_{#1}}(#2,#3)}
\newcommand{\neiz}[3]{\mathscr{N}_{#1,#2}}
\newcommand{\col}[3]{\Omega_{#2}(#3)(#1)}
\newcommand{\colall}[2]{\Omega(#2)(#1)}
\newcommand{\colalls}[1]{\Omega(#1)}

\newcommand{\DM}{\mathrm{DMK}}

\newtheorem{thm}{Theorem}[section]
\newtheorem{prop}[thm]{Proposition}
\newtheorem{lemma}[thm]{Lemma}
\newtheorem{cor}[thm]{Corollary}

\newtheorem{fact}{Fact}

\newnumbered{defi}[thm]{Definition}
\newnumbered{notation}[thm]{Notation}
\newnumbered{remark}{Remark}

\SelectTips {cm}{}

\title[Trivializing the circle in the framed little disks]%
{Homotopically trivializing the circle in the framed little disks}
\author{Gabriel C. Drummond-Cole}
\classno{55P48, 18D50 (primary), 81T40 (secondary).
 Please refer to {\tt http://www.ams.org/msc/} for a list of codes}
\extraline{This material is based upon work supported by the National Science Foundation under Award No. DMS-1004625.}

\begin{document}
\maketitle
\begin{abstract}
This paper confirms the following suggestion of Kontsevich. In the appropriate derived sense, an action of the framed little disks operad and a trivialization of the circle action is the same information as an action of the Deligne--Mumford--Knudsen operad.  This improves an earlier result of the author and Bruno Vallette.
\end{abstract}
\section{Introduction}
Configurations of points and disks on a surface have a long history of study in a number of different branches of mathematics.  In some situations, spaces of such configurations can parameterize operations in topological spaces, modules over a ring, or more generally objects in a symmetric monoidal category.  In this point of view, the number of points or disks keeps track of the number of inputs and outputs of the parameterized operation.  Kontsevich suggested~\cite{Kontsevich:T} that a specific relationship should hold between two such operation spaces.  

The first set of operation spaces is the framed little disks~\cite{Getzler:BVATDTFT}, the non-compact moduli space of genus zero surfaces with parameterized (marked) boundary.  The framed little disks generate a subcategory of a certain two-dimensional cobordism category, where the connected morphisms are genus zero surfaces from some number of ``incoming'' circle boundary components to a single ``outgoing'' circle, and so these operations arise as the genus zero, $\itn$-to-one operations of a two dimensional topological conformal field theory.  There are evident circle actions in this setting by rotating the boundary circles. 

The second set of operation spaces is the Deligne--Mumford--Knudsen spaces~\cite{Knudsen:PMSSCIISMGN}, which constitute a compactification of the moduli space of genus zero surfaces with marked points.  An action of these spaces can be considered as the genus zero, $\itn$-to-one operations of a different kind of field theory.  

The relationship between these two spaces of operations, at this genus zero, $\itn$-to-one level, is that an action of the framed little disks can be extended to an action of the Deligne--Mumford--Knudsen spaces if the action of the circle is homotopically trivial, and that the data of such a trivialization of the circle gives rise to an action of the Deligne--Mumford--Knudsen spaces (and homotopically no additional information). This paper works through this statement rigorously, using the language of operads. Such homotopical trivializations of the circle arise in practice. For example, the action of the framed little disks in the field theory determined by the category of sheaves on a smooth projective Calabi-Yau manifold has a homotopically trivial circle action, using Keller~\cite{Keller:OCHRSS}.

The main theorem is the following:
\begin{thm}\label{thm:mainthm}
The Deligne--Mumford--Knudsen genus zero operad $\Mbar$ is a model of the homotopy pushout of the following diagram of operads.
\[\xymatrix{
\thecircle\ar[r]\ar[d]&\trivialoperad\\\fld  
  }\]
Here $S^1$ is the circle, $\trivialoperad$ is the trivial operad, and $\fld$ is the operad of framed little disks.
\end{thm}
This implies the following result at the level of representations:
\begin{cor}
For any sufficiently cofibrant models $\widetilde{\fld}$ and $\widetilde{\Mbar}$ of the framed little disks and Deligne--Mumford--Knudsen operad, the space of $\widetilde{\fld}$-structures on a space $\spaceb$ with trivial $\widetilde{\fld}(1)$-action is weakly equivalent to the space of $\widetilde{\Mbar}$-structures on $\spaceb$.
\end{cor}

The main theorem is pleasant, but not unexpected. Different versions of this statement have been discussed by Costello as well as Kontsevich, and a rough outline of a partial argument is present in~\cite{Markarian:HBVD}, although there is no indication of the homotopy invariance there.  A cognate of this theorem was proven at the level of rational homology in~\cite{DotsenkoKhoroshkin:FRGB}~and~\cite{DrummondColeVallette:MMBVO}, relying on earlier work of Getzler~\cite{Getzler:TDTGEC,Getzler:OMSGZRS}.  This theorem is a nice improvement; using adjunctions along the lines of those in~\cite{SchwedeShipley:EMMC}, one can see that the topological story implies the same result at the level of integral, not just rational, homology.  

\section{Executive summary}
This section provides a structured overview of the proof, a roadmap of the remainder of the paper, as well as a few words on possible extensions.

The proof begins with the construction of two operads $\mbarone$ and $\mbartwo$ which are the pushouts of two different models of the diagram above. That is, the following are pushout diagrams.
\[\xymatrix{
\thecircle\ar[d]\ar[r]&\tld\ar[d]&&{\fldone}\ar[r]\ar[d] & {\tld}\ar[d]\\
\flda\ar[r]&\mbarone&&{\fld}\ar[r]&{\td}
}\]
Here $\fldone$ and $\thecircle$ are two different models of the circle, $\tld$ is a model of the trivial operad, and $\fld$ and $\flda$ are two different models of the framed little disks.\footnote{These and all other operads used in the paper are defined in one handy location in the next section.}

There are coherent inclusions of the first diagram into the second:
\[\xymatrix{
\thecircle\ar[d]\ar[r]\ar[dr]&\tld\ar@{=}[dr]\\
\flda\ar[dr]&\fldone\ar[d]\ar[r]&\tld\\
&\fld}\]
which induce a map $\mbarone\to\mbartwo$. The operad $\mbartwo$ is constructed from scratch and then shown in Section~\ref{sec:charm} to be the pushout of the above diagram.

We use two models like in this manner because the former has good homotopical properties while the latter is easy to connect to the Deligne--Mumford--Knudsen operad $\Mbar$. Unpacking the first part of this statement, we achieve the following proposition as a corollary to a more abstract theorem in Section~\ref{sec:hpoa}.
\begin{prop}\label{prop:mbaroneishpush}
The topological operad $\mbarone$ is a realization of the homotopy pushout of the diagram of Theorem~\ref{thm:mainthm}
\end{prop}
Unpacking the second, we describe a suboperad $\td_{\DM}$ of $\mbartwo$ and prove the following in Section~\ref{sec:isdm}.
\begin{prop}\label{prop:dmisdm}
The operads $\td_{\DM}$ and $\Mbar$ are isomorphic.
\end{prop}
\begin{prop}\label{prop:deformationretract}The space $\td_{\DM}(\finiteset)$ is a deformation retract of the space $\td(\finiteset)$.
\end{prop}
So we get a diagram like this:
\[
\entrymodifiers={+!!<0pt,\fontdimen22\textfont2>}\xymatrix{
\mbarone\ar[r]&\mbartwo\ar@<.75ex>@{.>}[r]^-\sim&\td_{\DM}\ar@<.75ex>[l]\ar@{<->}[r]^-\cong&\Mbar
}\]
where solid arrows are maps of operads. 

To complete the proof, it is only necessary to show that the induced map $\mbarone\to\mbartwo$ is a weak equivalence. It turns out to be easier to show the following equivalent statement, which is done in Section~\ref{sec:triv} and Appendix~\ref{appendix:homotopy}:

\begin{thm}\label{thm:toptheoremtwo}
The map of collections $\psi:\mbarone\to\mbartwo\to \Mbar$ is a weak equivalence.
\end{thm}
This does not mean that $\mbartwo$ is superfluous in the proof; if it were not there we would have no guarantee that the weak equivalence of collections of spaces $\longproj$ arose from a zigzag of weak equivalences of operads.

The proof of theorem~\ref{thm:toptheoremtwo} is an intricate calculation of planar geometry, using some techniques of classical algebraic topology.

There are two other appendices containing a model category theory background and a brief description of the slightly nonstandard conventions used for trees.

Some words on possible extensions are in order.  It would be nice to extend this result to the full cobordism category. That is to say, this theorem is about operadic algebra of genus zero surfaces. The cognate statement about properadic algebra of surfaces of arbitrary genus should also be true. The results of Section~\ref{sec:hpoa} should hold with little modification, but there seem to be some technical difficulties in generating the correct version of $\mbartwo$ in that setting.  Further, the argument for the equivalence of the two pushouts would almost certainly need to be completely replaced.  

It is also natural to wonder how much of this holds for the framed $\itn$-balls, for $\itn>2$, viewed as a relatively much smaller subcategory of the $\itn$-cobordism category.  There the results of Section~\ref{sec:hpoa} fail completely, as the actions of $SO(\itn)$ on the framed $\itn$-balls are not free, and it seems unlikely that the homotopy quotient and naive quotient should coincide.  

The presence of a unit in the framed little disks operad also creates problems for Section~\ref{sec:hpoa}.  The unit disrupts the filtration of the framed little disks by arity, so that the space of one little disk cannot be restricted to the circle alone.  This may only be a technical problem, but the unit also has the same problem as higher dimensional balls: the action of the circle is not free.  This probably implies a space of units in the homotopy pushout equivalent to $B\thecircle$.  It is not clear to the author how this should be reflected in (a modification of) the geometric compactification.

\section{Variations on the framed little disks}\label{sec:opzoo}
The purpose of this section is to define the various operads that will be used in the sequel. References for operadic algebra include~\cite{May:GILS,Smirnov:HTC,GetzlerJones:OHAIIDLS}.
\begin{defi}
The group $\affc$ is $\cC\rtimes GL_1\cC\cong \cC\rtimes\cC^\ast$, which acts on the complex plane $\cC$ by translation, dilation, and rotation.

In coordinates, if the action is taken to be rotation and dilation first, followed by translation, then the group action is 
\begin{equation*}(\cent_1,\radi_1)\comp(\cent_2,\radi_2)=(\cent_1+\radi_1\cent_2, \radi_1\radi_2).\end{equation*}  The inverse of $(\cent,\radi)$ is $(-\frac{\cent}{\radi},\frac{1}{\radi})$ and the identity is $(0,1)$.

The topological group $\affc$ can also be identified with the configuration space of a disk with a marked point on its boundary in the plane; the element $(\cent,\radi)$
in this presentation corresponds to the disk centered at $\cent$ with a marked point on its boundary at $\cent+\radi$.  In this context, the composition map involves scaling a disk up or down and rotating it.  The identity corresponds to the standard disk.
\end{defi}
\begin{defi}
Let $\finiteset$ be a finite set. The {\em configuration space of $\finiteset$ framed disks in the plane} is:
\begin{equation*}
\conf(\finiteset)=\{(\cent_\element,\radi_\element)_{\element\in \finiteset}\in (\cC\times \cC^\ast)^\finiteset:|\cent_\element-\cent_{\element'}|> |\radi_\element|+|\radi_{\element'}|\}.
\end{equation*} 
The condition ensures that every two disks in the plane are disjoint.

The {\em configuration space of $\finiteset$ framed disks or points in the plane} is:
\begin{equation*}
\conft(\finiteset)=\{(\cent_\element,\radi_\element)_{\element\in \finiteset}\in (\cC\times \cC)^\finiteset:|\cent_\element-\cent_{\element'}|> |\radi_\element|+|\radi_{\element'}|\}.
\end{equation*}   
Let $\comp$ denote the composition map $(\cent_1,\radi_1)\comp (\cent_2,\radi_2)\mapsto (\cent_1+\radi_1\cent_2,\radi_1\radi_2)$ of $\affc$; by abuse of notation, 
 we will use $\comp$ to denote as well its restriction to various related spaces as well. 
\end{defi}
\begin{defi}\label{defi:fld1}The operad $\fld$ of {\em framed little disks} consists of the following spaces:
$\fld(\finiteset)$ is the subset of $\conf(\finiteset)$ such that:
\begin{enumerate}
\item $|\cent_\element|+|\radi_\element|\le 1$, and
\item this inequality is strict unless $\cent_\element=0$.
\end{enumerate}
We will use the custom that $\fld(\emptyset)$ is empty.  Note that for $|\finiteset|>1$, the conditions imply that $|\cent_\element|+|\radi_\element|<1$.

Geometrically, the $\cent_\element$ correspond to the centers of disjoint little disks in the large disk of radius one in the plane, each having radius $|\radi_\element|$ and with a marked point at $\cent_\element+\radi_\element$.

There is an evident action of the symmetric group of the set $\finiteset$ on the factors of $\conf(\finiteset)$.  There are also composition maps derived from $\comp$; explicitly, we can compose the framed little disks $((\cent_{\element'},\radi_{\element'}))$ into the factor $(\cent_{\hat{\element}},\radi_{\hat{\element}})$ of $((\cent_{\element},\radi_{\element}))$ as follows. Using the shorthand $\disk_\element$ to refer to the pair $(\cent_\element,\radi_\element)$,
\begin{equation*}
\{\disk_\element\}_{\element\in\finiteset}\cyrc_{\hat{\element}}\{\disk_{\element'}\}_{\element'\in\finiteset'}=
\{\disk_\element\}_{\element\in\finiteset,\element\ne \hat{\element}}\sqcup\{\disk_{\hat{\element}}\comp\disk_{\element'}\}_{\element'\in\finiteset'}.
\end{equation*}
So we replace $(\cent_{\hat{\element}},\radi_{\hat{\element}})$ with the factors $(\cent_{\hat{\element}},\radi_{\hat{\element}})\comp (\cent_{\element'},\radi_{\element'})$.
\end{defi}
\begin{defi}\label{defi:fld2}The operad $\flda$ of {\em annulus-free framed little disks} consists of the suboperad of $\fld$ so that if $|\finiteset|=1$ then $\flda(\finiteset)$ consists only of points on the unit circle, $(0,\radi)$ with $|\radi|=1$.  This is stable under composition.
\end{defi}
\begin{defi}\label{defi:tdrtwo}Let $\finiteset$ be an ordered finite set.  For convenience, suppose $\finiteset=\{1,2,\ldots\}$. The subspaces $\tdrtwo(\finiteset)$ of $\flda(\finiteset)$ consist of those points of $\flda(\finiteset)$ where $\cent_2-\cent_1\in\rR_+$ and $\radi_\element\in\rR_+$ (the first condition is empty for $|\finiteset|<2$).
\end{defi}
\begin{defi}\label{defi:ann}We denote by $\fldone$ the arity one part of $\fld$.  That is, $\fldone(\finiteset)$ is $\fld(\finiteset)$ if $|\finiteset|=1$ and is empty otherwise.

The operad $\thecircle$ is the complex units under multiplication, concentrated in arity one. That is, $\thecircle(\finiteset)$ is the complex units if $|\finiteset|=1$ and is empty otherwise. The operad $\thecircle$ can be viewed as a suboperad of $\fldone\subset\fld$ or $\flda$.
\end{defi}
\begin{remark}
The space $\fld(\finiteset)$ is acted on by $\thecircle$ on the left and $(\thecircle)^\finiteset$ on the right. Since $\thecircle$ is a suboperad of $\fld$, this action is just by composition. Acting by $\complexunit$ on the left and $\prod \complexunit_\element$ on the right takes the collection $\{(\cent_\element,\radi_\element)\}$ to $\{(\complexunit\cent_\element,\complexunit\complexunit_\element\radi_\element)\}$. Geometrically, the action on the right corresponds to rotating each marked point along the circle boundary of the disk indexed by $\element$ by the argument of $\complexunit_\element$. The action on the left corresponds to rotating the entire configuration by the argument of $\complexunit$.
\end{remark}
\begin{lemma}
The composition $\tdrtwo(\finiteset)\to \flda(\finiteset)\to \thecircle\backslash\flda(\finiteset)/(\thecircle)^{\finiteset}$ is a homeomorphism.
\end{lemma}
\begin{proof}
Acting on $\flda$ on the left by some element of $\thecircle$ and simultaneously on the right by some element $(\thecircle)^\finiteset$ takes $\cent_2-\cent_1$ and $\radi_\element$ to $\rR_+$.  These elements of $\thecircle$ are $\frac{|\cent_2-\cent_1|}{\cent_2-\cent_1}$ and $\frac{|\radi_\element|(\cent_2-\cent_1)}{\radi_\element|\cent_2-\cent_1|}$.  Thus, there is a map from $\flda(\finiteset)$ to $\thecircle\times (\thecircle)^\finiteset$ picking out these elements, which gives a map from $\flda(\finiteset)$ to $\tdrtwo(\finiteset)$ by composing on the left and right by these elements of $\thecircle$.  This map is stable on the equivalence classes defining the quotient, so descends to an inverse of the map described above.
\end{proof}
\begin{defi} The {\em trivialized annuli operad} $\tld$ is the operad such that $\tld(\finiteset)$ is the subset of $\conft(\finiteset)$ such that
$|\cent_\element|+|\radi_\element|\le 1$ if $|\finiteset|=1$ and $\tld(\finiteset)$ is empty otherwise.  This operad is a ``trivialization'' of $\fldone$ and/or $\thecircle$ and contains both as suboperads.
\end{defi}
\begin{defi}
The operad $\mbarone$ is the pushout of the following diagram of operads, where the maps are inclusion:
\[\xymatrix{
\thecircle\ar[d]\ar[r]&\tld\\
\flda
}\]
\end{defi}
The three remaining operads we will define in this section ($\mbartwo$, $\td_{\DM}$, and $\Mbar$) are operads of decorated trees. That is, if $\operad$ denotes one of these three operads, then a point in $\operad(\finiteset)$ is a point in a product of configuration spaces indexed by the vertices of some $\finiteset$-tree. In each case, composition is some variant of the tree grafting operation, although for the first two it is slightly more complicated than that. Throughout, we use the notation and terminology of Appendix~\ref{appendix:trees} for trees.

\begin{defi}
Let $\vertex$ be a vertex of a non-planar unreduced rooted $\finiteset$-tree $\tree$.  If $\vertex$ is the root, then the marking set $\marking_\vertex$ is the subset of $\conft(\edgesof(\vertex))$ such that:
\begin{enumerate}
\item $0<|\radi_\element|$ if $\radi_\element$ is in a factor $(\cent_\element,\radi_\element)$ indexed by an external edge of $\tree$,
\item $\radi_\element=0$ if $\radi_\element$ is in a factor indexed by an internal edge of $\tree$,
\item $|\cent_\element|+|\radi_\element|\le 1$ for every factor in the product, and
\item this inequality is strict unless $\cent_\element=0$.
\end{enumerate}
If $\vertex$ is not the root, then $\marking_\vertex$ is the subset of $\conft(\edgesof(\vertex))/\affc$ such that 
\begin{enumerate}
\item $0<|\radi_\element|$ if $\radi_\element$ is in a factor indexed by an external edge of $\tree$, and
\item $\radi_\element=0$ if $\radi_\element$ is in a factor indexed by an internal edge of $\tree$.
\end{enumerate}
Here, the $\affc$ action is on all factors of $\conft(\edgesof(\vertex))$ by left multiplication with $\comp$.

Both marking sets are configuration spaces of disjoint little disks indexed by the external edges in centered at $\cent_\element$ of radius $|\radi_\element|$ with one marked point on the boundary circle at $\cent_\element+\radi_\element$ along with additional disjoint marked points at $\cent_\element$ indexed by the internal edges.

The marking for the root is such configurations in the disk; the marking for other vertices is such configurations in the plane up to conformal automorphisms of the plane.
\end{defi}
\begin{defi}
The {trivialized little disks operad} $\td$ consists of the sets (for now) $\td(\finiteset)$ where $\td(\finiteset)$ is
\[
\coprod_\tree \prod_{\vertex\in \verticesof(\tree)}\marking_\vertex.\]
Here the disjoint union is over nontrivial nearly stable $\finiteset$-trees $\tree$.
\end{defi}

\begin{figure}[hb!]
\begin{pspicture}(-5,0)(5,5.5)
\psline(0,0)(0,3)
\psline(0,1)(-1,2)
\psline(0,1)(1,2)
\psline(0,3)(-1,4)
\psline(0,3)(1,4)
\psline(-1,4)(-1,5)
\rput(0,3){$\bullet$}
\rput(0,1){$\bullet$}
\rput(-1,4){$\bullet$}
\pscircle[fillcolor=lightgray, fillstyle=solid](4,2){2}
\pscircle[fillcolor=white, fillstyle=solid](3.5,1.5){1}
\pscircle[fillcolor=white, fillstyle=solid](5,3){.5}
\rput(6,2){$\bullet$}
\rput(3.5,2.5){$\bullet$}
\rput(5,2.5){$\bullet$}
\rput(3.5,1.5){$1$}
\rput(5,3){$4$}
\rput(4,2.8){$\bullet$}
\rput(4,3.1){$\{2,3\}$}
\rput(-1,5.25){$2$}
\rput(1,4.25){$3$}
\rput(-1,2.25){$1$}
\rput(1,2.25){$4$}
\psline{->}(2,1.5)(.15,1.05)
\psline{->}(-1.4,4)(-1.1,4)
\psline{->}(-1.4,2.5)(-.1,3)
\psframe*[linecolor=lightgray](-1.5,5.5)(-4,3)
\psframe*[linecolor=lightgray](-1.5,0)(-4,2.5)
\pscircle[fillcolor=white, fillstyle=solid](-2.75,4.25){1}
\rput(-2.75,4.25){$2$}
\rput(-1.75,4.25){$\bullet$}
\pscircle[fillcolor=white, fillstyle=solid](-2.1,1.25){.3}
\rput(-2.1,1.25){$3$}
\rput(-2.2248,1.5228){$\bullet$}
\rput(-2.75,1.25){$\bullet$}
\rput(-2.75,1.55){$2$}
\end{pspicture}
\caption{An element of $\td(\{1,2,3,4\})$. We refer to each edge by the set of leaves above it, conflating $\element$ with $\{\element\}$ for ease.}
\end{figure}
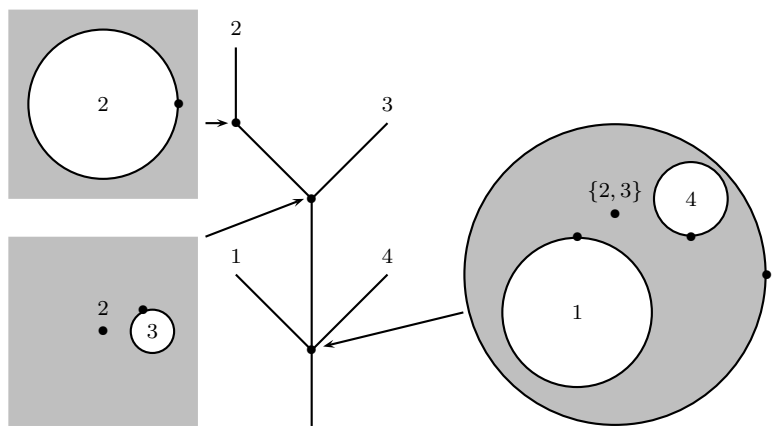

The composition in $\td$ is as follows: to glue a disk into a little disk in the plane up to conformal automorphism, scale the disk and glue it in using $\comp$.  This is independent of the conformal representative.  The other vertices of the tree involved do not change any decorations.  If the resulting configuration in the plane is a single point then forget the corresponding tree vertex. 

Given this composition structure, the tree with one leaf and one vertex decorated by $(0,1)$ is a unit for composition.  It is routine to verify that composition is associative, since $\comp$ is.

\begin{defi}
We define $\td_{\DM}(\finiteset)$ to be the subspace of $\td(\finiteset)$ characterized by the following.
\begin{enumerate}
\item If $|\finiteset|=1$, then $\td_{\DM}(\finiteset)$ is just the identity configuration $(0,1)$ decorating the $1$-corolla in $\td(\finiteset)$.
\item If $|\finiteset|>1$, then in the tree $\tree$ underlying any element of $\td_{\DM}(\finiteset)$, the successor vertex of each leaf and the root vertex must each be bivalent. Further, the decoration of the root vertex must be the single configuration $(0,0)$.
\end{enumerate}
\end{defi}
\begin{lemma}
$\td_{\DM}$ is a suboperad of $\td$.
\end{lemma}
\begin{proof}
Any compostion of two non-identity elements of $\td_{\DM}$ must involve the unstable grafting of a leaf of one to the root of the other.  The other root and the other leaves are not involved in this composition and remain bivalent after it.
\end{proof}

In Section~\ref{sec:charm}, we will topologize both $\td$ (and $\td_{\DM}$, using the subspace topology) and in general when we refer to them, we will be referring to their topologized versions.

Next we give a description of the so-called Deligne--Mumford spaces of genus $0$ nodal curves with marked points. This description, which is derived from~\cite{HinichVaintrob:ATSO}, is not the standard one, but is convenient for our purposes. We also prefer to call these spaces Deligne--Mumford--Knudsen spaces to emphasize Knudsen's contribution~\cite{Knudsen:PMSSCIISMGN} to the theory.
\begin{defi}
The collection $\msph$ is defined as follows:
\begin{equation*}\msph(\finiteset)=\left\{
\begin{array}{ll}
\emptyset& |\finiteset|<2\\
\emb(\finiteset,\cC)/\affc & |\finiteset|\ge 2
\end{array}
\right..\end{equation*}
Here $\emb(\finiteset,\cC)$ denotes embeddings of $\finiteset$, viewed as a discrete space, into $\cC$.
\end{defi}
\begin{defi}
The collection of {\em Deligne--Mumford--Knudsen sets} $\Mbar$ is defined as follows, using the notation of Appendix~\ref{appendix:trees}:
\begin{equation*}
\Mbar(\finiteset)=\coprod_{\tree\in\treesone{\finiteset}}\prod_{\vertex\in \verticesof(\tree)}\msph(\insof(\vertex))
\end{equation*}
That is, an element of $\Mbar(\finiteset)$ is a stable tree (no bivalent vertices) whose vertices are each decorated with an embedding of $\insof(\vertex)$ into $\cC$ up to the action of the affine group of $\cC$.
\end{defi}
\begin{defi}
The {\em realization} of $\prod \psplone_\vertex$ in $\Mbar(\finiteset)$ is a topological space with extra structure obtained as follows.  For each vertex, take a copy of complex projective line, marked at $\infty$ and by the points of $\psplone_\vertex$ up to $\affc$.  If there is an edge between two vertices, identify $\infty$ on one copy of complex projective line with the correct point of the other.  Remove the points $\finiteset$ corresponding to the leaves from this quotient, and give it the subspace topology.  The topological structure does not depend on any choices, and because everything is taken up to conformal automorphisms, this space has a conformal structure away from the identified points (which are called {\em nodal points}).
\end{defi}
\begin{defi}
A {\em contraction} from one point to another in $\Mbar(\finiteset)$ is a map of realizations which respects $\finiteset$. It is a homeomorphism except possibly at nodal points of the realization of the range, and so that the preimage of a nodal point is either a nodal point or a circle containing no nodal points.
\end{defi}
\begin{remark}
Such a map is called a contraction because topologically it contracts circles to nodal points. Combinatorially it might make more sense to call it a vertex expansion, since contracting a circle corresponds to expanding a vertex into two vertices.
\end{remark}
\begin{defi}
$\Mbar(\finiteset)$ is a topological space with a topology with a local basis as follows.  Let $\psplone= \prod \psplone_\vertex$ be an element of $\Mbar(\finiteset)$.  The configuration $\psplone_\vertex\in \msph(\insof(\vertex))$ is a subset of the complex projective line taken up to the action of $\affc$.  Let $\unei_\vertex$ be a bounded open subset of $\cC$ (taken up to $\affc$) whose closure does not contain any point of $\psplone_\vertex$.  The neighborhood $\neiz{\psplone}{\unei}{\vertex}$ consists of those points $y$ in $\Mbar(\finiteset)$ which have a contraction to $\psplone$ which is conformal on the preimage of $\unei_\vertex$.
\end{defi}
\begin{fact*}
$\Mbar$ has the structure of a topological operad, where the unit is the tree with one leaf and one vertex and partial composition is given by grafting trees.
\end{fact*}
\section{Characterizing a pushout}\label{sec:charm}
The purpose of this section is to show that $\mbartwo$ is the pushout of the diagram of topological operads
$\fld\gets\fldone\to\tld$.

\begin{prop}\label{prop:Xissetpushout}
As an operad of sets, $\td$ is the pushout of the trivialized annuli operad $\tld$ and the framed little disks $\fld$ over $\fldone\subset \tld$:
\begin{equation*}
\xymatrix{
{\fldone}\ar[r]\ar[d]&{\tld}\ar[ddr]\ar[d]\\
{\fld}\ar[drr]\ar[r]&{\td}\ar@{.>}[dr]\\
&&{\operad}
}
\end{equation*}
\end{prop}
The maps in the diagram are as follows.

The configuration space in $\td$ for a $\finiteset$-corolla is in canonical bijection with $\fld(\finiteset)$.  Composition of corollas in $\td$ agrees with composition in the framed little disks.  This discussion includes annuli of positive radius.

The annulus $(\cent,0)$ maps to the $1$-tree with two vertices, the root decorated by $(\cent,0)$ and the other vertex decorated by a single little disk $(\cent,\radi)$ up to $\affc$ (which is no information, as this is the quotient of $\affc$ by itself).  A straightforward check demonstrates that composition of two radius zero annuli, or of a radius zero annulus with a positive radius annulus  agrees with composition in $\td$ on both the right and the left.
\begin{fact}\label{fact:pushout}
The coproduct of two operads $\firstop$ and $\secondop$ in the category of sets (or topological spaces) can be realized as decorated trees whose vertices are decorated with elements of the individual operads $\firstop$ and $\secondop$ of the appropriate arity, and so that the set of decorations on two adjacent vertices always has one decoration from each of $\{\firstop,\secondop\}$. Composition is grafting of trees along with composition of $\firstop$ or $\secondop$ as necessary.

The pushout of two operads $\firstop$ and $\secondop$ along a third operad $\thirdop$ is the further quotient of the coproduct by the relation that one can change the image of an element of $\thirdop$ in $\firstop$ to the image of the same element in $\secondop$, contracting the tree as necessary using the composition in $\secondop$ (as well as the same relation with $\firstop$ and $\secondop$ interchanged).
\end{fact}
A version of this fact is presented in~\cite{GetzlerJones:OHAIIDLS}, although this version is slightly different.
\begin{proof}[of Proposition~\ref{prop:Xissetpushout}]
We begin by showing that the sets making up $\td$ agree with those of the pushout of the diagram as presented in the fact above. Since the trivialized annulus $(\cent,0)$ is equal to the composition $(\cent,\radi)\comp (0,0)$ for any choice $\radi$, the relations of the pushout imply that every element has a representative where all of the elements of $\tld$ are of the form $(0,0)$. If there is a vertex decorated by a non-trivialized annulus after performing this procedure, then we may compose in $\tld$ or $\fld$ and eliminate it. Then all vertices will be either
\begin{itemize}
\item at least trivalent or 
\item bivalent and decorated with the configuration $(0,0)$ with the exception that 
\item if the root is bivalent, it may be decorated with the configuration $(\cent,0)$.
\end{itemize} 

In fact, the internal vertices will alternate between these two types. Now any edge of an at least trivalent vertex that does not come from a leaf comes from a marking of the form $(0,0)$. This means that the marking corresponding to any such edge $(\cent,\radi)$ is equivalent to the marking $(\cent,\radi')$ for any $\radi'$ such that $0<|\radi'|\le|\radi|$ because $(0,0)$ annihilates the annulus $(0,\frac{\radi'}{\radi})$ on the left. So we lose no information by forgetting all such radii. Further, the total collection of markings on any vertex whose sucessor is bivalent and decorated with the configuration $(0,0)$ is equivalent to the same collection where every marking is acted on on the left by $(\cent,\radi)$, $\radi\ne 0$, since $(0,0)$ annihilates $(\cent,\radi)$ on the right. 

Since the vertices alternate, we lose no information by forgetting every internal bivalent vertex. At this point we essentially have the set making up $\td$. The vertices that are not root vertices and not bivalent are decorated by collections of points (corresponding to internal edges) and disks (corresponding to external edges) in the disk up to an overall left action of $\affc$, which is the same as this set in the plane. The root vertex is the same, except there is no left action. The bivalent vertices that are successors to leaves contain no further information.

It is easy to check that no further quotient is possible given the relations in the pushout and that composition and the maps involved agree.
\end{proof}
\begin{prop}\label{prop:topologyexistsandpushout}
There are topologies on the sets $\td(\finiteset)$, described below, which make $\td$ a topological operad, make the embeddings of the trivialized annuli and the framed little disks continuous, and realize $\td$ as the pushout in the category of topological operads.
\end{prop}
There is an approach to proving this proposition where one would use the topological version of Fact~\ref{fact:pushout}, but as we shall need to know explicit details about this topology for Propositions~\ref{prop:isamanifold}~and~\ref{prop:XDMisDM}, we instead construct the topology by hand.  

We begin by describing the topology of Proposition~\ref{prop:topologyexistsandpushout}.
We contain the marking set $\marking_\vertex$ into a larger ambient space.

\begin{defi}
Let $\vertex$ be a vertex of a non-planar unreduced rooted $\finiteset$-tree $\tree$.  If $\vertex$ is the root, then the {\em augmented marking space} $\augmarking_\vertex$ is the subset of $\conft(\edgesof(\vertex))$ such that:
\begin{enumerate}
\item $0<|\radi_\element|$ if $\radi_\element$ is in a factor $(\cent_\element,\radi_\element)$ indexed by an external edge of $\tree$,
\item $|\cent_\element|+|\radi_\element|\le 1$ for every factor in the product, and
\item this inequality is strict unless $\cent_\element=0$.
\end{enumerate}
If $\vertex$ is not the root, then $\augmarking_\vertex$ is the subset of $\conft(\edgesof(\vertex))/\affc$ such that 
$0<|\radi_\element|$ if $\radi_\element$ is in a factor indexed by an external edge of $\tree$.

The augmented marking space contains the marking set as the subspace with $\radi_\element=0$ for each factor indexed by an internal edge of $\tree$.
\end{defi}
Now fix a nontrivial nearly stable $\finiteset$-tree $\tree$ and fix an ordering of $\finiteset$.
We will define a set map
\begin{equation*}
\prod_{\vertex\in \verticesof(\tree)} \augmarking_\vertex\to \td
\end{equation*}
The part of $\td(\finiteset)$ with underlying tree $\tree$ looks like $\prod \marking_\vertex$, and such points will be taken to themselves by this map.

For $\psplone$ in $\td(\finiteset)$ with underlying tree $\tree$, let $\unei_\vertex$ be a neighborhood of $\psplone_\vertex$; then the image of $\prod \unei_\vertex$ will be an neighborhood of $\psplone$ set in our basis, called $\neiz{\psplone}{\unei}{\vertex}$.

The map works as follows.  Let $\pspltwo=\prod \pspltwo_\vertex$ be a point in $\displaystyle\prod_\tree\augmarking_\vertex$.  Let $\someedges$ be the set of internal edges $\edge$ of $\tree$ so that  $|\radi_\edge|>0$.  Then $\pspltwo$ will land in the component of $\td(\finiteset)$ whose underlying tree is the edge contraction of $\tree$ along the edges in $\someedges$.

To specify the map, we must provide maps from a product of $\augmarking_\vertex$ to a single $\marking_\vertex$ after the contraction.  The ordering of $\finiteset$ induces one on $\insof(\vertex)$; for ease, suppose $\insof(\vertex)=\{1,2,\ldots\}$.
If $\vertex$ is not the root\footnote{Let $\tilde{\pspltwo}_\vertex$ denote $\pspltwo_\vertex$ for $\vertex$ the root.}, then there is a unique representative $\tilde{\pspltwo}_\vertex$ in the $\affc$ orbit of $\pspltwo_\vertex$ with $\cent_1=0$, $\cent_2$ on the positive real line, and the Euclidean diameter of the union of the disks of center $\cent_\element$ and radius $|\radi_\element|$ in the plane equal to $\frac{1}{2}$. There is a unique marking with center $0$ and radius vector to $1$ if $\vertex$ has only one incoming edge (which must then be external and have a nonzero radius).  We can view $\tilde{\pspltwo}_\vertex$ as a configuration in the disk, rather than the plane.  Then the map from a product of $\augmarking_\vertex$ to $\marking_\vertex$ is just iterated $\comp$ of the $\tilde{\pspltwo}_\vertex$ along the factors of the product specified by the contraction edges.  The associativity of $\comp$ ensures that this is independent of the order of contraction.  

The point $\psplone$ has radius vector zero for all internal edges, so is taken to itself, as promised.

\begin{lemma}
\label{lemma:basis}
The set $\{\neiz{\psplone}{\unei}{\vertex}\}$ forms the basis for a topology which is independent of the ordering chosen on $\finiteset$.
\end{lemma}
\begin{proof}
It is clear that these sets cover $\td(\finiteset)$. The intersection of two sets
$\neiz{\psplone}{\unei}{\vertex}\cap \neiz{\psplone}{\unei'}{\vertex}$ contains $\neiz{\psplone}{\unei\cap\unei'}{\vertex}$ so it suffices to show that for any $\psplone\in \neiz{\psplnaught}{\unei}{\vertex}$ there exists a basis element $\neiz{\psplone}{\unei}{\vertex}$ centered at $\psplone$ contained therein.  Note that 
$\psplnaught$ lives in a product over the tree $\tree$ and $\psplone$ lives in the product over the tree $\tree_{\someedges}$, where $\someedges$ is some set of internal edges of $\tree$. 
For each $\unei_\vertex\subset \augmarking_\vertex$, for $\vertex\in\verticesof(\tree)$, consider the open subset ${\unei_\vertex}'$ specified to be those points of $\unei_\vertex$ so that factors $(\cent,\radi)$ indexed by edges in $\someedges$ have nonzero radius.

Now we must define open sets $\unei_\vertex\subset\augmarking_\vertex$ for each vertex $\vertex$ of the tree $\tree_\someedges$.  Such a vertex is an equivalence classes of vertices in $\tree$. Then the open set $\unei_{\{\vertex_{\edge_\element}\}}$ (where the subscript comprises the elements of such an equivalence class) corresponds to the composition of $\unei'_\vertex$ by $\comp$ along the factors indexed by contraction edges.  As before, we pick a unique normalized representative for the factors in $\augmarking_\vertex$ involved in a composition. 

It is an exercise to show that the resultant sets $\unei_\vertex$ are open, as desired. 

Changing the ordering on $\finiteset$ changes the normalization used in the composition map.  It is sufficient to show that if $\finiteset$ and $\hat{\finiteset}$ are the same set with different ordering, then every neighborhood $\neiz{\psplone}{\unei}{\vertex}$ contains a neighborhood $\widehat{\neiz{\psplone}{\unei'}{\vertex}}$.  The two normalizations differ by a rotation and a translation of less than $\frac{1}{2}$.  The open set $\unei$ contains a neighborhood of a pair $(\cent,0)$ into which the decoration on this vertex is to be composed; this neighborhood must contain the ball consisting of all pairs $(\cent',\radi')$ with $|\cent'-\cent|$ and $|\radi'|$ less than $\smallepsilon$, for some small $\smallepsilon$.  If the corresponding neighborhood in $\unei'$ is chosen inside the ball of radius, say, $\frac{\smallepsilon}{2}$, then the composition using the normalization from $\hat{\finiteset}$ will be contained in that from $\finiteset$. 
\end{proof}
\begin{proof}[that the operadic structure maps are continuous]
Now that the topology is defined, it should be shown that the various structure maps of the operad are continuous.  This is obvious in the case of the unit, and the symmetric group action is continuous because the topology is independent of the ordering chosen to define the topologizing map.

Next, if $\psplone=\pspltwo\cyrc_{\element} \psplthree$, then the preimage under $\cyrc_{\element}$ of $\neiz{\psplone}{\unei}{\vertex}$ should contain a neighborhood of $\pspltwo\times \psplthree$.  In the case of stable composition of $\pspltwo$ and $\psplthree$, this follows directly from the continuity of $\comp$, so we restrict our attention to the unstable case.  

In this case, there are special bivalent vertices $\conleaf$ of $\tree_\pspltwo$ and $\conroot$ of $\tree_\psplthree$ where the composition occurs.  The complement of these two vertices in the disjoint union of these trees' vertex sets is in canonical bijection with the vertex set of $\tree_\psplone$.

Call the vertex immediately below $\conleaf$ in $\tree_\pspltwo$ (and the corresponding vertex in $\tree_\psplone$) $\conlast$.  Without loss of generality, we assume that $\unei_\conlast$ contains an $\smallepsilon$-neighborhood of the decoration $\psplone_\conlast$ in some fixed normalization that doesn't depend on the decoration of the edge coming from $\conleaf$.

For the vertices corresponding to $\tree_\psplone$, $\neiz{\psplone}{\unei}{\vertex}$ already includes a choice of open set $\unei_\vertex$.  At the bivalent vertices, take $\unei_\conleaf=\augmarking_\conleaf$ (which is a single point, the image of $(0,1)$ under $\affc$) and $\unei_\conroot=\augmarking_\conroot$.

Define a map \[\displaystyle \prod_{\verticesof_\pspltwo} \augmarking_\vertex\times \prod_{\verticesof_\psplthree}\augmarking_\vertex\to \prod_{\verticesof_\psplone}\augmarking_\vertex\] which composes the decoration on $\conroot$ into the appropriate spot of the decoration on $\conlast$, keeping all other decorations the same.  We will show that performing this operation on points in our product of open sets lands in the product definining $\neiz{\psplone}{\unei}{\vertex}$ and that the following diagram commutes, completing the proof by demonstrating that $\neiz{\pspltwo}{\unei}{\vertex}\times\neiz{\psplthree}{\unei}{\vertex}$ is in the preimage of $\neiz{\psplone}{\unei}{\vertex}$.  
\[
\xymatrix{
\prod_{\verticesof_\pspltwo} \augmarking_\vertex\times \prod_{\verticesof_\psplthree}\augmarking_\vertex\ar[r]\ar[d]& \td\times \td\ar[d]^\comp\\
\prod_{\verticesof_\psplone}\augmarking_\vertex\ar[r]&\td}
\]
Since by assumption the composition is unstable, the radius of the decoration $(\cent_0,\radi_0)$ of $\pspltwo_\conlast$ corresponding to the edge from $\conleaf$ is $0$, and so for any point in the neighborhood $\unei_\conlast$, if $(\cent,\radi)$ is the cogent decoration we have $|\cent-\cent_0|+|\radi|<\smallepsilon$.  Then for any point $(\cent_\conroot,\radi_\conroot)$ in $\unei_\conroot\subset \affc$, we have
\begin{equation*}|\cent+\radi\cent_\conroot-\cent_0|+|\radi\radi_\conroot|\le |\cent-\cent_0|+|\radi|(|\cent_\conroot|+|\radi_\conroot|)\le|\cent-\cent_0|+|\radi|<\smallepsilon.\end{equation*}
This shows that the image of our constructed product neighborhood of $\pspltwo\times \psplthree$ is contained in $\prod \unei_\vertex$.  Commutativity of the diagram basically follows from associativity of $\comp$.  The marking on $\conleaf$ is always the identity, and never matters; there are a couple of easy extra cases where the radius of the salient markings on $\conlast$ or $\conroot$ is zero.
\end{proof}
\begin{proof}[that the embeddings are continuous]
For the framed little disks, there is nothing to check; the tree describing a point in the image of the framed little disks has no internal edges, so the topology on the image of $\fld(\finiteset)$ is just the subspace topology in $(\cC\times \cC)^{\finiteset}$, just as it is in the framed little disks.  For the annuli, it must be checked that the preimage of an open set around the image of a radius zero annulus is open in the operad of trivialized annuli.  Since such an annulus is described by a tree with one leaf and two vertices, the root and another vertex.  The decoration on the leaf edge is unique so the space $\prod \augmarking_\vertex$ is homeomorphic to the augmented marking space of the root.
A basis for the open sets around the trivial annulus $(\cent,0)$ in this space are formed by the subsets 
\begin{equation*}\{(\cent+\varone,\vartwo)\in\conft(\{\edge\}):|\cent+\varone|+|\vartwo|< 1, |\varone|,|\vartwo|<\smallepsilon\}.\end{equation*}  
This is open in $\tld(\{\edge\})$ as well.
\end{proof}
\begin{proof}[that $\td$ is the pushout]
Let $\operad$ be a topological operad that accepts topological operad maps from $\tld$ and $\fld$ that agree on $\fldone$.  Proposition~\ref{prop:Xissetpushout} indicates that there is a unique morphism of operads of sets from $\td\to\operad$ which factors both of these maps.  To prove the proposition, we must show that the induced set map is continuous.

Let $\pspltwo\in \operad$ and $\psplone\in \td$ in its preimage under the induced set map.  Then $\psplone$ has an underlying tree $\tree$ and can be written as a composition along the tree $\tree'$ obtained by inserting a vertex on every internal edge of $\tree$, where the new vertices are labeled by $(0,0)$ and the old vertices are labeled with minor modifications of the decorations of the corresponding vertices of $\tree$, as in the proof of Proposition~\ref{prop:Xissetpushout}.  Let $\verticesof$ denote the vertices of $\tree$ and $\newvertices$ the vertices of $\tree'$.  The diagram in figure~\ref{figure:bigcomm} commutes.

\begin{figure}[t!]
\begin{equation*}
\def\objectstyle{\displaystyle}
\xymatrix{
{
\prod_{\vertex\in \verticesof} \fld(\insof(\vertex))\times
\prod_{\vertex\in \newvertices\backslash\verticesof}\tld(\insof(\vertex))
}
\ar[rr]\ar[dd]&&{
\prod_{\vertex\in \newvertices} \operad
}\ar[dddddd]|{\text{compose along }\tree'}
\\\\
{\prod_{\vertex\in \newvertices} \td}
\ar[dd]|{\text{compose along edges out of vertices in}\newvertices\backslash\verticesof}\\\\
{\prod_{\vertex\in \verticesof}\td}\ar[dd]|{\text{compose along }\tree}\\\\
{\td}\ar[rr]
&&{\operad}
}
\end{equation*}
\caption{}
\label{figure:bigcomm}
\end{figure}
Let $\opnei$ be an open set containing $\pspltwo$.  The composition along the top and right side is continuous by assumption, so we can come up with open neighborhoods of any preimage of $\psplone$ in the product of $\fld$ and $\tld$.  Then we can push those maps down the left side of the diagram.  

Consider $\smallepsilon$-neighborhoods around fixed points $\fldel_\vertex\in\fld(\insof(\vertex))$ and around $(0,0)$.  The image of the product of these neighborhoods under the first two vertical maps on the left contains the $\smallepsilon$-neighborhood of the image of the appropriate composition of $\fldel_\vertex$ and $(0,0)$.  Then the composition of these two maps is open, so pushing forward the preimage of $\opnei$ we get an open neighborhood of $\psplone$ in $\prod \td$.
\end{proof}
\section{The relation to the Deligne--Mumford--Knudsen compactification}\label{sec:isdm}
The purpose of this section is prove Propositions~\ref{prop:dmisdm}~and~\ref{prop:deformationretract}, which state, respectively, that $\td_{\DM}$ and $\Mbar$ are isomorphic and that the inclusion of $\td_{\DM}$ into $\td$ is the inclusion of a deformation retract of collections (not operads).

The former follows from the following propositions:
\begin{prop}\label{prop:isamanifold}
For $|\finiteset|>1$, $\td_{\DM}(\finiteset)$ is locally homeomorphic to $\rR^{2|\finiteset|-4}$.
\end{prop}
\begin{prop}\label{prop:XDMisDM}
There is a bijective map of topological operads: \begin{equation*}\equivtodm:\td_{\DM}\to \Mbar.\end{equation*}
\end{prop}
\begin{proof}[of Proposition~\ref{prop:dmisdm}]
The space $\Mbar(\finiteset)\cong\Mbar_{0,|\finiteset|+1}$ is a manifold of dimension $2(|\finiteset|+1)-6=2|\finiteset|-4$ as well.  
Invariance of domain says that an injective continuous map from a space modelled locally by $\rR^{2|\finiteset|-4}$ to a manifold of the same dimension is a homeomorphism onto its image.  Therefore $\equivtodm$ is an isomorphism of topological operads.
\end{proof}

For Propositions~\ref{prop:deformationretract}~and~\ref{prop:XDMisDM}, the $|\finiteset|=1$ case is obvious, and for most of the rest of the section we shall assume $|\finiteset|>1$ without comment.

The following follows directly from the definition of $\td_{\DM}$:
\begin{lemma}\label{lemma:baseonXDM}Let $\psplone$ be in $\td_{\DM}(\finiteset)$ with underlying tree $\tree$ (and a fixed order on $\finiteset$).  For $\vertex$ a vertex of $\tree$, let $\psplone_\vertex$ be the decoration on $\vertex$.  Let $\unei_\vertex$ denote an open set containing $\psplone_\vertex$ in $\augmarking_\vertex$.  Then the images of sets of the form $\prod \unei_\vertex$ under the map
\[\prod \augmarking_\vertex\to \td(\finiteset)\]
land in $\td_{\DM}$ and form a basis for the subspace topology on $\td_{\DM}$.
\end{lemma}
\begin{defi}
Let $\vertex$ be a vertex of a non-planar unreduced rooted $\finiteset$-tree $\tree$.  We will define the {\em Deligne--Mumford--Knudsen marking space} $\dmmarking_\vertex$ of $\vertex$ as a subset of the marking space $\augmarking_\vertex$. 

If $\vertex$ is the root, $\dmmarking_\vertex$ is the singleton $\{(0,0)\}$. Otherwise, $\dmmarking_\vertex$ is the subspace of $\conft(\edgesof(\vertex))/\affc$ such that $\radi_\edge=0$ for every factor $(\cent_\edge,\radi_\edge)$ indexed by an edge coming from a bivalent vertex. 
\end{defi}
\begin{lemma}\label{lemma:bijectionofbase}
For a fixed $\psplone\in\td_{\DM}(\finiteset)$ and fixed ordering of $\finiteset$, the restriction of the map $\augmarking_\vertex\to\td_{\DM}$ of Section~\ref{sec:charm} to a map $\prod \dmmarking_\vertex\to\td_{\DM}$ is injective.
\end{lemma}
\begin{proof}
We want to show that a configuration in $\td_{\DM}$ uniquely determines the decorations on the vertices that were identified to make it.  Let $\tree$ be the underlying tree of $\psplone$. 

Let $\pspltwo$ be a point in the image of the map determined by $\psplone$; $\pspltwo$ determines an edge contraction set $\someedges$.   Let $\finiteset=\{\vertex_\element\}$ be a vertex in the contracted tree $\tree_\someedges$. We would like to recover the decorations on the vertices $\vertex_\element$ involved in the contraction from the decoration on the contracted vertex.  We will proceed downward through the set of contracted vertices, starting with the topmost vertices of $\finiteset$.

Assume we have recovered the decoration on every vertex above the vertex $\vertex$; we wish to recover the decoration $\psplone_\vertex$ of $\vertex$ from the contraction vertex decoration $\pspltwo_\finiteset$.  Essentially, we know that the decoration on $\finiteset$ is equal to a composition, so by inverting the composition map, which is just $\comp$, where we can, we will be able to recover the marking on $\vertex$.

By assumption, we know all the markings above $\vertex$, so we can compose the appropriate representatives to get a configuration $\psplone_\vertex^\edge$ of points in the disk for each incoming internal edge $\edge$ of $\vertex$ coming from $\finiteset$. There are only points because at the top level, every disk must have radius $0$.  Consider each $\psplone_\vertex^\edge$ with the normalization used for composition.  Because $\cent_1$ in each representative is $0$, the configuration $\psplone_\vertex^\edge$ has a marked point at $0$.  It has another marked point at $\cent_1\ne 0$.  In order that $(\cent,\radi)\comp \psplone_\vertex^\edge=(\cent',\radi')\comp \psplone_\vertex^\edge$, we must have 
$\cent+0\radi=\cent'+0\radi$ so $\cent=\cent'$ and $\cent+\radi\cent_1=\cent'+\radi'\cent_1$ so $\radi=\radi'$.  
This shows that we can uniquely factor $\pspltwo_\finiteset$ as the composition of some $\psplthree$ with the various $\psplone_\vertex^\edge$.  This gives us a center and radius vector in $\psplthree$ for each incoming edge of $\vertex$.  Now we want to decompose $\psplthree$ as the composition of some $\psplthree'$ and $\psplone_\vertex$.  We know which centers and radii in $\psplthree$ sitting in the standard disk came from $\psplone_\vertex$, including $\cent_0$ and $\cent_1$, the image of $0$ and a positive real.  We can also measure the Euclidean diameter of the union of the disks involved.  So setting $(\cent,\radi)\comp \psplone_\vertex$ to be the points and distances we know, we get the equations
\begin{equation*}\cent + 0\radi = \cent_0,\ \cent_1\in \cent+\radi\rR_+,\end{equation*} and the diameter is $\frac{|\radi|}{2}$.
These uniquely specify $\cent$ and $\radi\ne 0$ and we can compose on the left by the inverse to $(\cent,\radi)$ to obtain $\psplone_\vertex$.
\end{proof}
\begin{lemma}\label{lemma:baseisconty}
The map $\prod \dmmarking_\vertex\to\td_{\DM}$ is continuous.
\end{lemma}
\begin{proof}
Let $\tree$ be the tree over which the product in the domain is taken.  A variation of Lemma~\ref{lemma:basis} for $\td_{\DM}$ shows that it suffices to show that for $\psplone$ in the (open) image of this map and $\neiz{\psplone}{\unei}{\vertex}$ a neighborhood of $\psplone$ contained in this image, the preimage of $\neiz{\psplone}{\unei}{\vertex}$ is a neighborhood of the preimage of $\psplone$, which is unique by Lemma~\ref{lemma:bijectionofbase}.  The underlying tree $\tree_\someedges$ of $\psplone$ is obtained from $\tree$ by contracting some edges, and the restriction of the map
\[\prod_{\vertex\in \verticesof(\tree)} \dmmarking_\vertex\to\td_{\DM}\]
to an appropriate open subspace factors as
\[\xymatrix{\text{subspace of }\prod_{\tree} 
{{\dmmarking}_\vertex}
\ar[r]
\ar[dr]&\prod_{\tree_\someedges}
{{\dmmarking}_\vertex}\ar[d]
\\&\td_{\DM}
}
\]
where the horizontal map is composition of normalized representatives with $\comp$, which is continuous, yielding the result.
\end{proof}
\begin{proof}[of Proposition~\ref{prop:isamanifold}]
By Lemmas~\ref{lemma:bijectionofbase}~and~\ref{lemma:baseisconty}, every point in $\td_{\DM}(\finiteset)$ has a neighborhood that looks like $\prod \dmmarking_\vertex$ for some tree $\tree$ with a bivalent root vertex and bivalent successor vertices for every leaf. For ease, forget these to get a stable tree $\tree'$.  By using a different $\affc$ parameterization for the non-bivalent vertices, for example the one where $\cent_0$ is sent to $0$ and $\cent_1$ to $1$, it is easy to see that $\dmmarking_\vertex$ is a manifold with $2\intinsof(\vertex)+\extinsof(\vertex)-2$ complex parameters.  $\displaystyle \sum_\vertex \intinsof(\vertex)$ is the total number of internal edges, which is one less than the total number of vertices.  $\displaystyle \sum_\vertex \extinsof(\vertex)$ is $|\finiteset|$.  Then $\prod \dmmarking_\vertex$ has total real dimension $2|\finiteset|-4$, as desired.
\end{proof}
\begin{proof}[of Proposition~\ref{prop:deformationretract}]
Fix a point $\psplone$ in $\td(\finiteset)$ with underlying tree $\tree$.  Precompose and postcompose it with the trivialized annulus $(0,0)$ at every leaf and the root.  This gives a continuous map to $\td_{\DM}(\finiteset)$, which is the retraction.  
Because it does not change a tree whose leaf and root vertices are bivalent and whose root vertex is decorated with $(0,0)$, it is the identity on $\td_{\DM}(\finiteset)$.

The homotopy is a map $\homotopy:\td(\finiteset)\times [0,1]\to \td(\finiteset)$; it is given by precomposing at each leaf and postcomposing at the root with the (possibly trivialized) annulus $(0,\timez)$.  This is clearly the retraction map at $\timez=0$ and the identity at $\timez=1$.  It fixes $\td_{\DM}(\finiteset)$ as before.
\end{proof}
\begin{remark*}
The homotopy does not respect the operad structure in any intermediate stage.  Its purpose is just to ensure that the induced map from the homology of $\td_{\DM}$ to that of $\td$ is an isomorphism.
\end{remark*}
Now we will prove Proposition~\ref{prop:XDMisDM}.
\begin{defi}
Fix a point $\psplone$ in $\td_{\DM}(\finiteset)$.  Forget the bivalent vertices in the underlying tree of $\psplone$.  The decorations on each remaining vertex are configurations of points in the plane up to $\affc$ indexed by the incoming edges of the vertex.  By definition, such a configuration is a point of $\msph(\insof(\vertex))$, and the product of these points is a point in $\Mbar$.  This suffices to define
\[\equivtodm:\td_{\DM}\to \Mbar.\]
\end{defi}
\begin{proof}[that $\equivtodm$ is an operad map]
Composition in $\td_{\DM}$ grafts underlying trees, forgetting the two internal bivalent vertices that this creates, and preserves all other decorations.  $\equivtodm$ is compatible with this process, and thus is a map of operads.
\end{proof}
\begin{proof}[that $\equivtodm$ is bijective]
For an element in $\Mbar(\finiteset)$, insert a vertex on each external edge of the underlying tree.  Keep the decorations on each old vertex, mark the new root with $(0,0)$, and each new leaf vertex with the unique possible marking.  This constitutes an inverse map.
\end{proof}
\begin{proof}[that $\equivtodm$ is continuous]
Fix a point $\psplone$, which we will consider as being in both $\td_{\DM}(\finiteset)$ and $\Mbar(\finiteset)$.  Let $\neiz{\psplone}{\unei}{\vertex}$ be a neighborhood of $\psplone$ in $\Mbar$.  We must show that there is a neighborhood $\neiz{\psplone}{\unei'}{\vertex}$ in $\td_{\DM}$ contained in $\neiz{\psplone}{\unei}{\vertex}$.  Choose an ordering on $\finiteset$ and fix the normalization with $\cent_1$ at $0$, $\cent_2$ on the positive real axis and diameter $\frac{1}{2}$ for each decoration.  So we have fixed a configuration of pairs $\{(\cent_\edge,0)\}$ in $\conft(\insof(\vertex))$ for each vertex $\vertex$.

Let $\bigR> 1$ be a number so that the disk of radius $\bigR$ centered at $0$ contains $\unei_\vertex$ for each vertex (in this normalization), and choose $\smallepsilon$ to be small enough so that the disk of radius $(\bigR+1)\smallepsilon$ centered at $\cent_\edge$ is contained in the complement of $\unei_\vertex$ for every edge $\edge$ in $\insof(\vertex)$ for each vertex.  Then let $\unei'_\vertex$ consist of configurations made of pairs $(\cent_\edge', \radi_\edge')$ so that $|\cent_\edge'-\cent_\edge|<\smallepsilon$ and $|\radi_\edge'|<\smallepsilon$.

We must show that $\equivtodm(\neiz{\psplone}{\unei'}{\vertex})$ is contained in $\neiz{\psplone}{\unei}{\vertex}$.  Any point in $\neiz{\psplone}{\unei'}{\vertex}$ has a conformal subsurface at each vertex that looks like the disjoint union of a disk of radius one with some disks of radius $(\bigR+1)\smallepsilon$ removed and annuli indexed by $\insof(\vertex)$ with outer radius $\bigR\radi_\edge'$ and inner radius $\radi_\edge'$ (in the case $\radi_\edge'>0$. Composing the disk configuration into the corresponding annulus (which is part of the surface corresponding to the next vertex), we get a conformal subsurface containing $\unei_\vertex$.  There is an easy special case when $\radi_\edge'=0$.
\end{proof}
%
%
%
%
\section{Homotopy pushouts of operads}\label{sec:hpoa}
This somewhat abstract section describes pushouts of operads $\operad$ by monoids over $\operad(1)$ in the very specific setting that $\operad(1)$ is a group that acts freely on $\operad(\itn)$ for $\itn>1$, and shows that it is relatively easy to find homotopy models for such pushouts.
\begin{defi}\label{defi:scube}
Let $\finiteset$ be a finite set. The deleted $\finiteset$-cube category $\Scube$ has as its objects proper subsets of $\finiteset$ and morphisms inclusions.
\end{defi}
\begin{defi}
Let $\functora$ and $\functorb$ be two functors from the discrete category $\finiteset$ to a category $\category$ with small colimits. Let $\nattran$ be a natural transformation $\functora\to \functorb$.  Consider the functor from $\Scube$ to $\category$ which takes $\finitesubset\subset \finiteset$ to 
\[\prod_{\element\in \finitesubset} \functorb(\element)\times\prod_{\element \notin \finitesubset}\functora(\element)\]
with maps from $\nattran$.  The colimit of this functor is denoted $\cubecol{\finiteset}{\functora}{\functorb}$; it comes equipped with a map to $\prod_\finiteset \functorb(\element)$.
\end{defi}
\begin{lemma}\label{lemma:cubecofibration}
Let $\functora$ and $\functorb$ be two functors from $\finiteset$ to cofibrant spaces and $\nattran$ be a natural cofibration.  Then the induced map $\cubecol{\finiteset}{\functora}{\functorb}\to \prod_\finiteset \functorb(\element)$ is a cofibration.
\end{lemma}
\begin{proof}
We will proceed by induction. If $|\finiteset|=1$, the statement is obvious. Suppose $\finiteset=\finiteset'\sqcup \{\element\}$. Then $\cubecol{\finiteset}{\functora}{\functorb}$ is equal to the pushout of the following diagram.
\[
\xymatrix{
\functora(\element)\times \cubecol{\finiteset'}{\functora}{\functorb}\ar[d]\ar[rr]&& \functora(\element)\times\prod_{\element'\in\finiteset'}\functorb(\element')
\\
\functorb(\element)\times\cubecol{\finiteset'}{\functora}{\functorb}}
\]
By induction, the map $\cubecol{\finiteset'}{\functora}{\functorb}\to \prod_{\element'\in\finiteset'}\functorb(\element')$ is a cofibration.

Then by Proposition~\ref{prop:modelcatfact}, the induced map $\cubecol{\finiteset}{\functora}{\functorb}\to \functorb(\element)\times \prod_{\element'\in\finiteset'}\functorb(\element')$ is a cofibration. This is the desired result.
%
%
%
\end{proof}
\begin{defi}\label{defi:treesthing}
Fix a collection $\sspace$, a space $\spacea$, and a nonempty ordered finite set $\finiteset$. Let $\trees{\finiteset}{\itk}(\sspace,\spacea)$ be the space:
\[
\coprod_{\tree\in\trees{\finiteset}{\itk}}\prod_{\vertex\in \verticesof(\tree)}\sspace(\edgesof(\vertex))\times \prod_{\edgesof(\tree)}\spacea.
\]
\begin{figure}[hb!]
\begin{pspicture}(3,4)
\psline(2,1)(0,3)
\psline(2,3)(1,2)
\psline(2,0)(2,1)(3,2)
\rput(0,3.2){$2$}
\rput(2,3.2){$3$}
\rput(3,2.2){$1$}
\rput(2,1){$\bullet$}
\rput(1,2){$\bullet$}
\rput(.5,2){$\sspace(2)$}
\rput(1.5,1){$\sspace(2)$}
\rput(2.2,.5){$\spacea$}
\rput(2.45,1.65){$\spacea$}
\rput(1.65,1.65){$\spacea$}
\rput(1.45,2.65){$\spacea$}
\rput(.65,2.65){$\spacea$}
\end{pspicture}
\caption{A schematic picture of one component of $\trees{\{1,2,3\}}{2}(\sspace,\spacea)$.}
\end{figure}

Let $\subspacegroup$ be a subspace of $\spacea$; then let $\trees{\finiteset}{\itk}^\subspacegroup(\sspace,\spacea)$ be
\[\coprod_{\tree\in\trees{\finiteset}{\itk}}\prod_{\vertex\in \verticesof(\tree)}\sspace(\edgesof(\vertex))\times\prod_{\edgesof^\ext(\tree)}\spacea\times \cubecol{\edgesof^\inter(\tree)}{\subspacegroup}{\spacea},
\]the subspace of $\trees{\finiteset}{\itk}(\sspace,\spacea)$ consisting of points where at least one of the factors in the subproduct $\prod_{\edgesof^\inter(\tree)}\spacea$ is actually in $\subspacegroup$.
\end{defi}
\begin{defi}
Fix $\sspace$, $\spacea$, and $\subspacegroup$ as in Definition~\ref{defi:treesthing}.
Suppose that for any internal edge $\edge$ of a tree $\tree$ (between vertices $\vertex$ and $\vertex'$), there are {\em edge collapse maps} $\sspace(\insof(\vertex'))\times \subspacegroup\times \sspace(\insof(\vertex))\to \sspace(\insof(\{\vertex,\vertex'\}))$, where $\{\vertex,\vertex'\}$ is the contraction vertex of $\tree_\edge$.  Suppose these maps, along with canonical isomorphisms, extend the assignment $\treesone{\finiteset}\to\coprod \trees{\finiteset}{\itk}(\sspace,\subspacegroup)$ to a functor $\functor$ from trees to spaces.  That is, the edge collapse maps satisfy associativity and identity constraints.

In this case, for $|\finiteset|>1$, define $\col{\finiteset}{0}{\sspace,\subspacegroup,\spacea}$ to be $\trees{\finiteset}{1}(\sspace,\spacea)\cong \spacea\times\sspace(\finiteset)\times\spacea^\finiteset$. Suppose we have defined $\col{\finiteset}{\itk}{\sspace,\subspacegroup,\spacea}$ so that it accepts a map from $\trees{\finiteset}{\itk+1}^\subspacegroup(\sspace,\spacea)$. Define $\col{\finiteset}{\itk+1}{\sspace,\subspacegroup,\spacea}$ to be the pushout of the following diagram:
\[\xymatrix{
\trees{\finiteset}{\itk+1}^\subspacegroup(\sspace,\spacea)\ar[rr]\ar[d]&&\col{\finiteset}{\itk}{\sspace,\subspacegroup,\spacea}\\
\trees{\finiteset}{\itk+1}(\sspace,\spacea)}.\]
Then the morphisms of the functor $\functor$ induce a map from $\trees{\finiteset}{\itk+2}^\subspacegroup(\sspace,\spacea)$ to $\col{\finiteset}{\itk+1}{\sspace,\subspacegroup,\spacea}$, which is well-defined by functoriality (this also establishes the map in the base case).

Let $\colall{\finiteset}{\sspace,\subspacegroup,\spacea}$ be the colimit of $\col{\finiteset}{\itk}{\sspace,\subspacegroup,\spacea}$, which stabilizes at some finite $\itk$ dependent on $|\finiteset|$.

If $|\finiteset|=1$, let $\colall{\finiteset}{\sspace,\subspacegroup,\spacea}$ be $\trees{\finiteset}{0}(\sspace,\spacea)=\spacea$.
\end{defi}
\begin{defi}
A map $\mapofcollections$ between two collections is a {\em cofibration of pointed collections} if it has the left lifting property with respect to all morphisms of pointed collections that are Serre fibrations on each finite set.  We call a pointed collection {\em cofibrant as a pointed collection} if the morphism from $\unitcollection$ is a cofibration of pointed collections.
\end{defi}
We will use Berger-Moerdijk's model category structure~\cite{BergerMoerdijk:AHTO}:
\begin{thm}\label{thm:modeloperads}
There is a model category structure on topological operads where the fibrations (weak equivalences) are morphisms of operads which are Serre fibrations (induce isomorphisms of all homotopy groups) on each finite set.
\end{thm}
\begin{thm}\label{thm:pushoutshape}
Let $\operad$ be an operad with $\operad(\emptyset)$ empty and for $|\finiteset|=1$, $\operad(\finiteset)$ an abelian group.  We will suppress $\finiteset$ and refer to $\operad(\finiteset)$ for $|\finiteset|=1$ as $\subspacegroup$.  Let $\topmon$ be a topological monoid containing $\subspacegroup$ as a submonoid (with the subspace topology).  Assume that for $|\finiteset|>1$, $\operad(\finiteset)$ is a free $\subspacegroup-\subspacegroup^{\finiteset}$ bimodule and that the quotient splits so that $\subspacegroup\times \left(\subspacegroup\backslash\operad(\finiteset)/\subspacegroup^{\finiteset}\right)\times \subspacegroup^\finiteset\cong \operad(\finiteset)$ as topological spaces with actions of $\subspacegroup$ and $\subspacegroup^\finiteset$ (the action of $\sS_\finiteset$ on the quotient and on the factors of $\subspacegroup^\finiteset$ may be twisted by some induced action on the $\subspacegroup$ factors, including the single factor on the left).
Define $\sspace$ as the collection with $\sspace(\finiteset)=\subspacegroup\backslash\operad(\finiteset)/\subspacegroup^{\finiteset}$.  We will view $\sspace(\finiteset)$ as a subspace of $\operad(\finiteset)$, using the identity in $\subspacegroup$.

Then $\colalls{\sspace,\subspacegroup,\topmon}$ is the pushout of the diagram $\operad\gets \subspacegroup\to \topmon$ in the category of operads.
\end{thm}
\begin{proof}
This is again a version of Fact~\ref{fact:pushout}. The use of $\sspace$ rather than $\operad$ and the twisted symmetric group action corresponds to taking the necessary quotients of the coproduct.
\end{proof} 
The composition in this presentation is by grafting of trees and composition in $\topmon$. For the reader's ease, let us describe the twisted symmetric group action, which is slightly tricky, explicitly. 

Consider a point $\psplone=\prod \psplone_\vertex\times \prod \psplone_\edge$ in $\colall{\finiteset}{\sspace,\subspacegroup,\topmon}$ with underlying tree $\tree$, and a permutation $\permut\in \sS_\finiteset$.  First, $\sS_\finiteset$ acts on $\finiteset$-trees naturally, so $\permut(\psplone)$ will have underlying tree $\permut \tree$.  

The permutation $\permut$ induces an isomorphism between the incoming edges of a vertex in $\tree$ and the incoming edges of the corresponding vertex in $\permut\tree$.

An element $\psplone_\vertex$ of $\sspace(\insof(\vertex))$, viewed in $\operad(\insof(\vertex))$, is taken by this action to an element of $\operad(\insof(\vertex))$ of the form
\[\gpel_\vertex \times \psplone_\vertex' \times \prod_{\edge\in\insof(\vertex)}\gpel_\edge\] 
for some $\psplone_\vertex'\in \sspace(\insof(\vertex))$ and $\gpel_\vertex, \gpel_\edge\in \subspacegroup$.  If $\psplone_\vertex$ is the label of $\vertex$, let $\psplone_\vertex'$ be the label on $\permut(\vertex)$.  For a fixed internal edge $\edge$ labeled by $\psplone_\edge$, where $\edge$ is the outgoing edge of $\vertex$, label $\permut(\edge)$ with $\gpel_\edge \psplone_\edge\gpel_\vertex^{-1}$.  This is the twisting of the induced action mentioned above. Perform a similar operation for external edges, replacing missing elements of $\subspacegroup$ with the identity.

\begin{prop}\label{prop:onetotwo}
Let $\operad$ and $\topmon$ be a pair as in Theorem~\ref{thm:pushoutshape}, let $\operad'$ and $\topmon'$ be another such pair, and suppose there are maps between them forming the diagram
\[\xymatrix{\operad\ar[d]&\subspacegroup\ar[r]\ar[l]\ar[d]&\topmon\ar[d]\\
\operad'&\subspacegroup'\ar[r]\ar[l]&\topmon'
  }
  \]
so that $\subspacegroup\to\subspacegroup'$ is an isomorphism so we shall suppress the notation $\subspacegroup'$, $\topmon\to \topmon'$ is a weak equivalence of monoids, and $\operad\to\operad'$ is a weak equivalence of operads so that $\operad(\finiteset)\to \operad'(\finiteset)$ is $\subspacegroup-\subspacegroup^\finiteset$-equivariant.

Assume that $\subspacegroup$ and $\topmon$, and $\sspace(\finiteset)=\subspacegroup\backslash\operad(\finiteset)/\subspacegroup^\finiteset$ are cofibrant spaces, and likewise for $\topmon'$ and $\operad'$, and that $\subspacegroup\to\topmon$ and $\subspacegroup\to \topmon'$ are cofibrations of spaces.

Then the induced map of pushouts is a weak equivalence.
\end{prop}
\begin{proof}
The long exact sequence of homotopy groups and the five lemma imply a weak equivalence from $\sspace(\finiteset)\to\sspace'(\finiteset)$.  Since products preserve weak equivalences of spaces, we have a weak equivalence $\col{\finiteset}{0}{\sspace,\subspacegroup,\topmon}\to\col{\finiteset}{0}{\sspace',\subspacegroup,\topmon'}$ for $|\finiteset|>1$.  Proposition~\ref{prop:reedyprop} and the fact that coproducts also preserve weak equivalences of cofibrant spaces imply that we have weak equivalences 
\[\trees{\finiteset}{\itk}^{\subspacegroup}(\operad,\topmon)\to \trees{\finiteset}{\itk}^{\subspacegroup}(\operad',\topmon')\] and 
\[\trees{\finiteset}{\itk}(\operad,\topmon)\to \trees{\finiteset}{\itk}(\operad',\topmon').\] 

The spaces $\trees{\finiteset}{\itk}^{\subspacegroup}(\operad,\topmon)$ and $\trees{\finiteset}{\itk}^{\subspacegroup}(\operad',\topmon')$ are cofibrant and the maps to $\trees{\finiteset}{\itk}(\operad,\topmon)$ and $\trees{\finiteset}{\itk}(\operad',\topmon')$ are cofibrations by Corollary~\ref{cor:prodcofibration} and Lemma~\ref{lemma:cubecofibration}.  Then by induction, if 
\[\col{\finiteset}{\itk}{\sspace,{\subspacegroup},\topmon}\to \col{\finiteset}{\itk}{\sspace',{\subspacegroup},\topmon'}\]
is a weak equivalence of cofibrant spaces, so is 
\[\col{\finiteset}{\itk+1}{\sspace,{\subspacegroup},\topmon}\to \col{\finiteset}{\itk+1}{\sspace',{\subspacegroup},\topmon'},\]
using the nearly direct category structure of the pushout diagram (see Proposition~\ref{prop:reedyprop}). 

Then, by the same proposition, $\col{\finiteset}{\bullet}{\sspace,\subspacegroup,\topmon}\to \col{\finiteset}{\bullet}{\sspace',\subspacegroup,\topmon'}$ is a weak equivalence of cofibrant telescopes so the weak equivalence passes to the colimit.
\end{proof}
\begin{thm}\label{thm:Spitzweck}
If $\firstop\gets\secondop\to\thirdop$ is a diagram of operads which are cofibrant as pointed collections, $\secondop\to \firstop$ is a cofibration of operads, and $\secondop\to \thirdop$ is a weak equivalence, then the induced map $\firstop\to \firstop\sqcup_\secondop \thirdop$ is a weak equivalence of operads and $\thirdop\to \firstop\sqcup_\secondop \thirdop$ is a cofibration of pointed collections.
\end{thm}
This is parts of Theorem~3.2 and Proposition~3.6 of~\cite{Spitzweck:OAMMCM}.

The following construction is classical.
\begin{thm}[\cite{BoardmanVogt:HIASTS,BergerMoerdijk:BVROMMC}]\label{thm:bm}
There exists a functor $\wfunctor$ on operads $\operad$ which are cofibrant as pointed collections so that $\wfunctor\operad$ is cofibrant as both an operad and a pointed collection, and a natural weak equivalence $\wfunctor\to \id$.  The functor $\wfunctor$ takes cofibrations of pointed collections to maps that are cofibrations of operads. 
\end{thm}
\begin{lemma}
There exists an operad $\einf$ satisfying the following properties:
  \begin{enumerate}
  \item $\einf(\finiteset)$ is a point if $|\finiteset|=1$
  \item $\einf(\finiteset)$ is contractible
  \item $\einf(\finiteset)$ is cofibrant as a pointed collection.
\end{enumerate}
\end{lemma}
\begin{proof}
There is such an operad where $\einf(\finiteset)$ is (some functorial version of) the total space of the universal bundle over the classifying space of $\sS_\finiteset$, with structure maps induced by the concomitant maps among the various $\sS_\finiteset$.  There is a standard choice satisfying the first condition.
\end{proof}
\begin{lemma}\label{lemma:einf}
Let $\operad$ be an operad so that the unit map is a cofibration of spaces.  Consider the operad $\operad\times \einf$ with $(\operad\times\einf)(\finiteset)=\operad(\finiteset)\times\einf(\finiteset)$, using the product of the structure maps of $\operad$ and $\einf$.  Then the projection map $\operad\times\einf\to\operad$ is a weak equivalence and $\operad\times\einf$ is cofibrant as a pointed collection.
\end{lemma}
\begin{proof}
The weak equivalence is obvious.  The (not necessarily equivariant) morphism space from $\operad(\finiteset)$ into a fibration with an action of $\sS_\finiteset$ form a fibration of spaces with $\sS_\finiteset$ action (by conjugation).  $\einf(\finiteset)$ has the left lifting property with respect to this fibration.  Because $\sS_\finiteset$-spaces form a Cartesian closed category, we get the desired lifting property on the product operad.  There is a special argument when $|\finiteset|=1$.
\end{proof}

\begin{prop}\label{prop:firstpushout}
Let $\firstop\gets\secondop\to\thirdop$ be a diagram of operads, so that all objects are cofibrant pointed collections and all maps are cofibrations of pointed collections. There is an operad $\replaced{\thirdop}$ equipped with maps $\secondop\to \replaced{\thirdop}\to \thirdop$ which factor $\secondop\to \thirdop$ into a cofibration of pointed collections followed by a weak equivalence and so that $\firstop\gets \secondop\to \replaced{\thirdop}$ is a homotopically correct model of $\firstop\gets\secondop\to\thirdop$.
\end{prop}
\begin{proof}
Let $\interoppushout$ be the pushout of the diagram $\secondop\gets \wfunctor\secondop\to \wfunctor\thirdop$.  This clearly comes with maps $\secondop\to \replaced{\thirdop}\to\thirdop$ which factor $\secondop\to\thirdop$.  By Theorem~\ref{thm:Spitzweck}, the map $\secondop\to\replaced{\thirdop}$ is a cofibration of pointed collections.  By Theorems~\ref{thm:Spitzweck}~and~\ref{thm:bm}, the map $\replaced{\thirdop}\to\thirdop$ is a weak equivalence.  

The pushout of $\firstop\gets\secondop\to\replaced{\thirdop}$ is canonically isomorphic to the pushout of $\firstop\gets \wfunctor\secondop\to\wfunctor\thirdop$, and it suffices to show that this latter diagram is a homotopically correct model.  

By Theorem~\ref{thm:bm}, $\wfunctor\firstop\gets \wfunctor\secondop\to \wfunctor\thirdop$ is a cofibrant model of the diagram, with realization $\oprealization$.  By Theorems~\ref{thm:Spitzweck}~and~\ref{thm:bm}, the induced map from $\oprealization$ to the pushout of the diagram $\firstop\gets \wfunctor\firstop\to\oprealization$ is a weak equivalence. But $\firstop\gets\wfunctor\firstop\to\oprealization$ is canonically isomorphic to $\firstop\gets \wfunctor\secondop\to\wfunctor\thirdop$, and the weak equivalence from $\oprealization$ to the pushout of this diagram arises from the map of diagrams from  $\wfunctor\firstop\gets \wfunctor\secondop\to \wfunctor\thirdop$ to  $\firstop\gets \wfunctor\secondop\to \wfunctor\thirdop$.
\end{proof}
Now we can combine Propositions~\ref{prop:onetotwo}~and~\ref{prop:firstpushout} to achieve the main theorem of this section.
\begin{thm}~\label{thm:abstracthpush}
Let $\operad$ and $\topmon$ be a pair as in Theorem~\ref{thm:pushoutshape}.  Assume that $\subspacegroup\to\topmon$ is a cofibration between cofibrant pointed spaces (rather than just a cofibration between cofibrant spaces) Then the diagram $\operad\gets\subspacegroup\to\topmon$ is a homotopically correct model for itself.
\end{thm}
\begin{proof}
By Lemma~\ref{lemma:einf} and Proposition~\ref{prop:firstpushout}, the diagram $\operad\times\einf\gets \subspacegroup\to \replaced{\topmon}$ is a homotopically correct model ($\subspacegroup$ and $\topmon$ are trivially cofibrant as pointed collections).  By Proposition~\ref{prop:onetotwo} the induced map of pushouts from its pushout to the pushout of $\operad\gets\subspacegroup\to\topmon$ is a weak equivalence.
\end{proof}
Proposition~\ref{prop:mbaroneishpush} now follows. The substitution of $\trivialoperad$ for $\tld$ makes no difference as there is a map $\tld\to \trivialoperad$ making all diagrams commute.
%
%
%
\section{Homotopy trivializing the circle in the framed little disks}\label{sec:triv}
In this section, we connect the two pushouts that we have constructed, proving Theorem~\ref{thm:toptheoremtwo}, that the map $\longproj:\mbarone\to\mbartwo\to\Mbar$ is a weak equivalence.
\begin{lemma}\label{lemma:folklore}[Folklore]
Let $\spacemap:\spacea\to \spaceb$ be a map and $\cover$ an open cover of $\spaceb$ which is closed under finite intersections.  If $\spacemap:\spacemap^{-1}(\Unei)\to \Unei$ is a weak equivalence for all $\Unei\in \cover$ then $\spacemap$ is a weak equivalence.
\end{lemma}
There is a proof of this lemma in~\cite{May:WEQ}.
\begin{defi}
Let $\spacea$ and $\spaceb$ be spaces.  A {\em deformation retraction over $\spaceb$} is the data of a deformation retraction with projection map $\retract$ from $\spacea$ onto $\spaceb$ with specified section and homotopy $\homotopy$ so that the following diagram commutes:
  \[
\xymatrix{\interval\times \spacea\ar[r]^\homotopy\ar[d]_{\proj_2}&\spacea\ar[d]^\retract\\\spacea\ar[r]_\retract&\spaceb}
\]
\end{defi}
\begin{lemma}~\label{lemma:DRoY}
Let $\spacemap:\spacea\to \spaceb$ be the retraction map of a deformation retraction over $\spaceb$. Then for any subset $\unei$ of $\spaceb$, the restriction of $\spacemap$ to $\spacemap^{-1}(\unei)$ is a weak equivalence.
\end{lemma}
\begin{proof}
The homotopy of the deformation retraction restricts to $\spacemap^{-1}(\unei)$.
\end{proof}
\begin{defi}
Let $\psplone$ be a point in $\Mbar(\finiteset)$ with underlying stable tree $\tree$, and let $\finiteset'$ be a subset of $\finiteset$.  Let $\vertex$ be a vertex of $\tree$ below all the leaves in $\finiteset'$.  Choose an $\affc$ representative $\psplone_\vertex$ of the decoration of $\vertex$.  Each marked point in $\psplone_\vertex$ is indexed by some edge $\edge$.  The {\em center of mass} $\cog_{\vertex,\finiteset'}$ of $\finiteset'$ in $\psplone_\vertex$ is a point in $\cC$ defined to be the weighted average of the marked points of $\psplone_\vertex$, where the weight of the marking indexed by $\edge$ is the number of elements of $\finiteset'$ above $\edge$.  This is $\affc$-equivariant.
\end{defi}
\begin{defi}
A {\em stable subset} of the finite set $\finiteset$ is a proper subset with at least two elements.
\end{defi}
\begin{defi}
Let $\setofsubsets=\{\finiteset_\iti\}$ be a set of stable subsets of $\finiteset$.  The {\em $\setofsubsets$-stratum} of $\Mbar(\finiteset)$ consists of points whose underlying tree has an edge $\edge_i$ for every $i$ so that the set of leaves above $\edge_i$ is precisely $\finiteset_\iti$.  The {\em codimension} of a stratum is $|\setofsubsets|$. The {\em open $\setofsubsets$-stratum} is the complement of all strata of higher codimension in the $\setofsubsets$-stratum.
\end{defi}
\begin{defi}
Let $\setofsubsets=\{\finiteset_\iti\}$ be a set of stable subsets of $\finiteset$.  The open $\setofsubsets$-set in $\Mbar(\finiteset)$ consists of points so that for each edge $\edge$ of a point's underlying tree and each stable subset $\finiteset_\iti$, the set of leaves above $\edge$ either is disjoint from, is contained in, or contains $\finiteset_\iti$.

Let $\psplone$ be in the open $\setofsubsets$-set, with underlying tree $\tree$.  Suppose we have normalized the decorations $\psplone_\vertex$ on each vertex of $\tree$ with respect to $\affc$ in some way.  A {\em simultaneous division of $\psplone$} along $\setofsubsets$ is a set of radii $\radi_{\vertex,\iti}$ for each pair consisting of a vertex $\vertex$ of $\tree$ and an index of $\setofsubsets$ such that $\finiteset_\iti$ is a union of the leaves above a set of incoming edges $\{\edge_{\iti,\itj}\}$ of $\vertex$.  These radii should satisfy the following:
  \begin{enumerate}
  \item the circle in $\cC$ centered at $\cog_{\vertex,\finiteset_\iti}$ separates those marked points in $\insof(\vertex)$ corresponding to the edges $\edge_{\iti,\itj}$ from all other marked points and $\infty$, and
  \item for fixed $\vertex$, any two such circles are disjoint.
  \end{enumerate}
\end{defi}
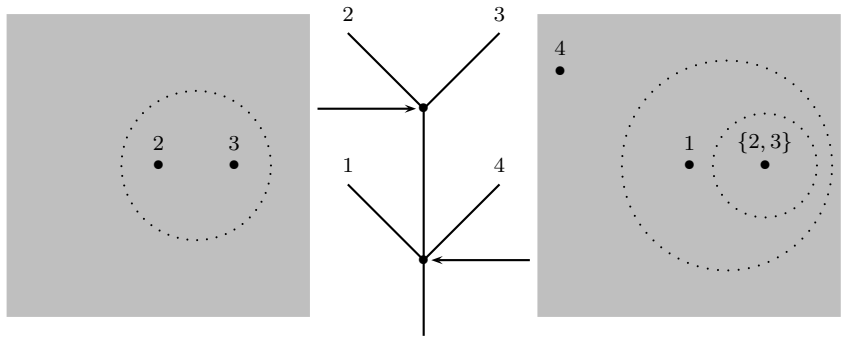
\begin{figure}[hb!]
\begin{pspicture}(-5,0)(5,4.5)
\psline(0,0)(0,3)
\psline(0,1)(-1,2)
\psline(0,1)(1,2)
\psline(0,3)(-1,4)
\psline(0,3)(1,4)
\rput(0,3){$\bullet$}
\rput(0,1){$\bullet$}
\rput(-1,4.25){$2$}
\rput(1,4.25){$3$}
\rput(-1,2.25){$1$}
\rput(1,2.25){$4$}

\psframe*[linecolor=lightgray](-5.5,.25)(-1.5,4.25)
\psframe*[linecolor=lightgray](1.5,.25)(5.5,4.25)
\rput(-3.5,2.25){$\bullet$}
\rput(-3.5,2.55){$2$}
\rput(-2.5,2.25){$\bullet$}
\rput(-2.5,2.55){$3$}
\pscircle[linestyle=dotted](-3,2.25){1}
\rput(3.5,2.25){$\bullet$}
\rput(3.5,2.55){$1$}
\rput(4.5,2.25){$\bullet$}
\rput(4.5,2.55){$\{2,3\}$}
\rput(1.8,3.5){$\bullet$}
\rput(1.8,3.8){$4$}
\pscircle[linestyle=dotted](4.5,2.25){.7}
\pscircle[linestyle=dotted](4,2.25){1.4}
\psline{->}(-1.4,3)(-.1,3)
\psline{->}(1.4,1)(.1,1)

\end{pspicture}
\caption{A simultaneous division of a point in $\Mbar(\{1,2,3,4\})$ along $\{\{2,3\}, \{1,2,3\}\}$}
\end{figure}

\begin{defi}Let $\tree$ be a $\finiteset$-tree.  The $\tree$-neighborhood $\nei_\tree$ in $\Mbar(\finiteset)$ consists of points which
can be simultaneously divided along the stable subsets $\intedges(\tree)$, and do not belong to codimension one strata except possibly the $\{\edge\}$-strata, for $\edge\in \intedges(\tree)$.
\end{defi}
\begin{lemma}
The $\tree$-neighborhoods are open and cover $\Mbar$.
\end{lemma}
\begin{proof}
Having a simultaneous division is an open condition, and the $\tree$-stratum is contained in the $\tree$-neighborhood.
\end{proof}
\begin{defi}
Let $\tree$ be a stable $\finiteset$-tree.
Consider 
 \[\prod_{\vertex\in \verticesof(\tree)}\tdrtwo(\edgesof(\vertex))\times \prod_{\intedges(\tree)}\tld\times \prod_{\extedges(\tree)}\{(0,0)\}\]
as a subspace of $\trees{\finiteset}{|\verticesof(\tree)|}(\tdrtwo,\tld)$.  Then it comes with a map to $\mbarone$, which is isomorphic to $\colalls{\tdrtwo,\thecircle,\tld}$ by Theorem~\ref{thm:pushoutshape}, and thus to $\Mbar$.  We will call this map $\longmap$.  
Let $\mbarthree_\tree$ be the preimage under $\longmap$ of the $\tree$-neighborhood, so that the following diagram commutes:
  \[
    \xymatrix{\mbarthree_\tree\ar[r]\ar@/^1.5pc/[rr]^\longmap&\mbarone\ar[r]^\longproj&\Mbar}\]
\end{defi}
We will use $\mbarthree_\tree$ as an intermediate space to show that $\longproj^{-1}(\nei_\tree)\to \nei_\tree$ is a weak equivalence.
\begin{lemma}\label{lemma:giveswe1}
$\longmap:\mbarthree_\tree\to\nei_\tree$ is a deformation retraction over $\nei_\tree$. %
\end{lemma}
\begin{proof}[that there is a section $\secty$ of $\longmap$]
Let $\psplone$ be a point in $\nei_\tree$ with underlying tree $\tree_\psplone$.  We will define a point $\secty(\psplone)$ over it in $\mbarthree_\tree$. We must specify a decoration in $\tdrtwo$ for each vertex and $\tld$ for each edge.

Since $\psplone$ can be simultaneously divided along $\intedges(\tree)$, the tree $\tree$ can be contracted to the tree $\tree_\psplone$.  Any edge of $\tree$ which is not so contracted is decorated by $0\in \tld$.  Then it will suffice to consider the case when $\tree_\psplone$ has only one vertex; in the general case we can use the same procedure for each vertex of $\tree_\psplone$.

So in the case where $\tree_\psplone$ has one vertex, we can consider $\psplone$ in $\emb(\finiteset,\cC)/\affc$, so a set of $\finiteset$-indexed disjoint points up to $\affc$.  Modify this configuration to a new set of points by adding the centers of mass $\cog_\edge$ of each edge $\edge$ of $\tree$.  These may coincide with one another and with the points indexed by $\finiteset$.

Choose an $\affc$-normalization for this set.  We will define a function $\rmin:\edgesof(\tree)\to [0,\infty)$.  If $\edge$ is a leaf, then $\rmin(\edge)=0$.  Assume that $\rmin$ has been defined on every edge $\edge'$ which is a subset of $\edge$.  Then $\rmin(\edge)$ is the infimum radius so that the disk centered at $\cog_\edge$ of radius $\rmin(\edge)$ contains the disks centered at $\cog_{\edge'}$ of radius $\rmin(\edge')$.  This set is nonempty because $\psplone$ is in the $\tree$-neighborhood.

Next, we will define another function $\rmax:\edgesof(\tree)\to (0,\infty]$.    For each edge, once $\rmin$ and $\rmax$ are both defined, let $\rsmall(\edge)$ and $\rbig(\edge)$ be, respectively, the harmonic and arithmetic mean of $\rmin(\edge)$ and $\rmax(\edge)$. 

For $\edge$ the root of $\tree$, let $\rmax(\edge)=\infty$. Now assume $\rmax$ is defined for an edge $\edge$, and that the edges $\edge_1,\ldots, \edge_\ita$ are the edges of $\tree$ which have $\edge$ as their output.  For $\iti$ in $1,\ldots, \ita$, let 
\[\disty_\iti=\min\left\{
|\cog_{\edge_\iti}-\cog_\edge|-\rmin(\edge_\iti)+\rsmall(\edge),\{|\cog_{\edge_\iti}-\cog_{\edge_\itj}|-\rmin(\edge_\iti)-\rmin(\edge_\itj)\}_{\iti\ne \itj}
\right\}\]
It is easy to see that $\disty_\iti>0$.  Let \[\rmax(\edge_\iti)=\rmin(\edge_\iti)+\frac{\disty_\iti}{2}.\]
Now consider the pairs \[
\left(\frac{\cog_{\edge_\iti}-\cog_\edge}{\rsmall(\edge)},\frac{\rbig(\edge_\iti)}{\rsmall(\edge)}\right).
\]
The set of these for $\edge_\iti$ an incoming edge of $\vertex\in \verticesof(\tree)$ is an element $\psplone'_\vertex$ in $\conf(\insof(\vertex))$.  The choice depends on the normalization, but only up to rotation.  There is a rotation $\angleof_\vertex$ so that $(0,e^{\imaginaryunit\angleof_\vertex})\comp \psplone'_\vertex$ is in $\tdrtwo(\insof(\vertex))$.  Let this be the decoration on the vertex $\vertex$.  Let the decoration on the edge $\edge$ be 
\[\frac{\rsmall(\edge)}{\rbig(\edge)}e^{-\imaginaryunit\angleof_\vertex}.\]
Translation does not affect any of the radii or a difference in centers of mass.  Rotation only affects centers of mass, and is accounted for by $\angleof_\vertex$.  Dilation acts on all radii and each center of mass in the same manner, so these formulae do not depend on the $\affc$-representative used to generate them.

If we compose the decorations of $\secty(\psplone)$ using the map $\comp$, we eventually arrive at a set of points
\[\frac{\cog_\ell-\cog_\rootv}{\rsmall(\rootv)}\]
indexed by the leaves $\ell$ of $\tree$, where $\rootv$ is the root of $\tree$.  The set of these is conformally equivalent to the set $\{\cog_\ell\}$ by the action of $(\cog_\rootv,\rsmall(\rootv))\in\affc$.  Each $\cog_\ell$ is the center of mass of one point, and so this shows that $\secty$ is a section of $\longmap$.

To see continuity within an open stratum, consider that $\cog_\edge$ varies continuously as one moves around an open stratum; then $\rmin(\edge)$ (in a fixed $\affc$ representative) can be realized as the distance from $\cog_\edge$ to the complement of some disks which depend on other $\cog_\edge$ and earlier $\rmin(\edge)$.  Then $\rmin$ varies continuously by induction.  The same is true for $\rmax$, so $\rsmall$ and $\rbig$ also vary continuously.  

Now suppose the sequence $\psplone_\itk$ in an open stratum approaches a degeneration $\psplone$ in a higher codimension stratum in $\Mbar$.  If $\edge$ is an edge that does not correspond to a node in $\psplone$, then by appropriately renormalizing, the decoration on $\edge$ in $\psplone_\itk$ approaches that in $\psplone$ as in the previous paragraph.  If, on the other hand, $\edge$ corresponds to a node in $\psplone$, then let $\edge_0$ and $\edge_1$ be two incoming edges of $\edge$.  Normalize by placing $\cog_{\edge_0}$ and $\cog_{\edge_1}$ at $0$ and $1$, respectively; $\rmin(\edge)$ will stay bounded and bounded away from zero, while $\rmax(\edge)$ will blow up.  Then $\frac{\rsmall(\edge)}{\rbig(\edge)}$ will approach $0$, and we will achieve the desired edge decoration in the limit.  

Switching normalizations, normalize $\psplone_\itk$ to make $\rsmall(\edge)=1$ (this can be done unless $\edge$ is a leaf), put $\cog_\edge$ at $0$, and rotate so that $\angleof_\edge=0$.  In this normalization, the decoration on the vertex $\edge$ consists of pairs $(\cog_{\edge_\iti}, \rbig(\edge_\iti))$.  In the limit, if $\edge$ corresponds to a node, then we must be careful; it is involved in two trees and has different radius values in each.  In the upper tree, the one pertinent to the case we are considering, we have finite $\rmin(\edge)$, so we can use the same normalization, and we get the limits of the sets of pairs.  That is, $\cog_{\edge_\iti}$ is insensitive to whether points coincide in the limit, and $\rbig(\edge_\iti)$ depends only on $\rmax(\edge)$, which approaches $\infty$ as $\psplone_\itk\to \psplone$, and $\rmin(\edge_\itj)$, which approaches $0$ in those cases where $\edge_\itj$ also corresponds to a node in $\psplone$.
\end{proof}
\begin{defi}[Construction of a homotopy for Lemma~\ref{lemma:giveswe1}]
As in the construction of $\secty$, we will define the homotopy explicitly only in the top stratum.  In other strata, where some edges must be decorated with $0\in\td$, we will use multiple copies of the homotopy for various subtrees glued together along these edges.

We can describe the portion of $\mbarthree_\tree$ above the top stratum in a different presentation.  Up to $\affc$, its image in $\Mbar$ is a set of disjoint points in $\cC$ indexed by the leaves of $\tree$ which is realized by composition of the various vertex and edge decorations.  Therefore, we can realize a point of $\mbarthree_\tree$ as a set of disjoint points and circles; for an edge decorated by the point $\radi$, we draw circles $\disk_\edge^1$ and $\disk_\edge^\radi$ (we will refer to them collectively as $\disk_\edge$) corresponding to the image of the circles of radius $1$ and $|\radi|$.  At the leaves, the inner ``circle'' has radius $0$ and coincides with the marked point, and at the root the outer ``circle'' has radius $\infty$.  Finally, the location of the marked point on each boundary circle is determined by the normalizations involved.  We have said enough to justify the following description:
\begin{lemma}
Let $\finiteset$ be a finite set, and fix a $\finiteset$-tree $\tree$.  The top stratum $\mbarthreetop_\tree$ of $\mbarthree_\tree$ is homeomorphic to the subset of $\conf(\edgesof(\tree)\backslash\{\text{root}\})\times\conf(\edgesof(\tree))$ consisting of configurations $\prod (\cent_\edge,\radi_\edge)\times \prod (\hat{\cent}_\edge,\hat{\radi}_\edge)$ so that
\begin{enumerate}
\item For a fixed edge, the disk centered at $\cent_\edge$ of radius $\radi_\edge$ contains the disk centered at $\hat{\cent_\edge}$ of radius $\radi_\edge$,
\item all radii are real,\footnote{This does not mean that the marked points of the various configurations are all in the positive real direction from the corresponding centers.  Rather, because the normalizations determine the arguments of the marked points, we lose no information by forgetting them, and this choice is a convenient one.}
\item if $\edge$ contains $\edge'$ then the disk centered at $\cent_\edge$ of radius $\hat{\radi}_\edge$ contains the disk centered at $\cent_{\edge'}$ of radius $\radi_{\edge'}$, and
\item $\edge$ is a leaf if and only if $\hat{\radi}_\edge=0$.
\end{enumerate}
\end{lemma}
We will now describe an $\affc$-equivariant homotopy over the top stratum of $\Mbar$ from the identity to $\secty$.  The beginning and end of the homotopy, at any point in the domain of the homotopy, have the same underlying configuration of points, which we arbitrarily normalize once and for all.  They also have the same nesting pattern for the various circles; that is, for a particular point $\psplone$, if the disk bounded by one such circle is contained in the disk bounded by another at time $0$, then it is also so contained at time $1$ (with some obvious exceptions if two disks coincide).  In particular, since every circle contains some marked points, which are in the same position at the beginning and end of the homotopy, the open disks bounded by $\disk_\edge(\psplone)$ and $\disk_\edge(\secty\lambda \psplone)$ have nonempty intersection.

Focussing on one particular circle $\disk_\edge$, we have a specified center and radius at time $0$ and $1$, $\cent_0,\cent_1$, $\radi_0$, and $\radi_1$.  We will extend these to $\cent_\timez$ and $\radi_\timez$ for all $\timez\in \interval$. 

Consider the subset $\overlappedset$ of points $(\cent_0,\radi_0,\cent_1,\radi_1)\subset \cC^2\times \rR_+^2$ so that $|\cent_0-\cent_1|<\radi_0+\radi_1$. We call such pairs of disks {\em properly overlapping}.  We will define a map $\homotopypart:\overlappedset\times \interval\to \cC\times\rR_+$, $\homotopypart=(\cent_\timez,\radi_\timez)$.

First define \[\angleof(\cent_0,\radi_0,\cent_1,\radi_1)=\arccos\left(\frac{\radi_0^2+\radi_1^2-|\cent_0-\cent_1|^2}{2\radi_0\radi_1}\right).\]
Because of the conditions, the quantity inside the arccosine is strictly greater than $-1$, so $\angleof$ (we will suppress the arguments) is either in $[0,\pi)$ or of the form $\realvariable\imaginaryunit$ for some positive number $\realvariable$.  $\arccos$ is a bijection between this domain and range.

Now, for $\angleof\ne 0$, define
\[\cent_\timez=\frac{\cent_1\radi_0\sin(\timez\angleof)+\cent_0\radi_1\sin((1-\timez)\angleof)}{\radi_0\sin(\timez\angleof)+\radi_1\sin((1-\timez)\angleof)},\qquad \radi_\timez=\frac{\radi_0\radi_1\sin(\angleof)}{\radi_0\sin(\timez\angleof)+\radi_1\sin((1-\timez)\angleof)};\]
If $\angleof=0$, let
\[\cent_\timez=\frac{\cent_1\radi_0\timez+\cent_0\radi_1(1-\timez)}{\radi_0\timez+ \radi_1(1-\timez)},\qquad \radi_\timez=\frac{\radi_0\radi_1}{\radi_0\timez+\radi_1(1-\timez)}.\]  This is easily continuous by the squeeze theorem.
\end{defi}
\begin{defi}
Let $\totallyadmissiblesubset$ be the subset of $\mbarthreetop_\tree\times\mbarthreetop_\tree$ consisting of tuples
\[
\prod (\cent_{\edge,0},\radi_{\edge,0})\times \prod (\hat{\cent}_{\edge,0},\hat{\radi}_{\edge,0})
\times
\prod (\cent_{\edge,1},\radi_{\edge,1})\times \prod (\hat{\cent}_{\edge,1},\hat{\radi}_{\edge,1})
\]
so that 
\begin{enumerate}
\item $(\cent_{\edge,0},\radi_{\edge,0},\cent_{\edge,1},\radi_{\edge,1})$ is in $\overlappedset$,  
\item likewise, if $\edge$ is not a leaf, $(\hat{\cent}_{\edge,0},\hat{\radi}_{\edge,0},\hat{\cent}_{\edge,1},\hat{\radi}_{\edge,1})$ is in $\overlappedset$, and
\item for $\leaf$ a leaf, $\cent_{\leaf,0}=\cent_{\leaf,1}$.
\end{enumerate}
Define a homotopy $\homotopy:\totallyadmissiblesubset\times\interval\to(\cC\times \rR_+)^{\edgesof(\tree)\backslash\{\text{root}\}}\times(\cC\times \rR_+)^{\edgesof(\tree)}$ by taking
$\{\cent_{\edge,\iti},\radi_{\edge,\iti};\hat{\cent}_{\edge,\iti},\hat{\radi}_{\edge,\iti}\},\timez$ to
\[\{\homotopypart_\timez(\cent_{\edge,0},\radi_{\edge,0},\cent_{\edge,1},\radi_{\edge,1});
\homotopypart_\timez(\hat{\cent}_{\edge,0},\hat{\radi}_{\edge,0},\hat{\cent}_{\edge,1},\hat{\radi}_{\edge,1})\}\]
where $\homotopypart_\timez(\hat{\cent}_{\leaf,0},\hat{\radi}_{\leaf,0},\hat{\cent}_{\leaf,1},\hat{\radi}_{\leaf,1})$, which was not defined, is interpreted as $(\hat{\cent}_{\leaf,0},\hat{\radi}_{\leaf,\iti})$, which is independent of $\iti$.
\end{defi}
The continuity of $\homotopypart$ implies the continuity of $\homotopy$ on $\totallyadmissiblesubset$.
\begin{prop}\label{prop:homotopyconstruction}
The image of the homotopy $\homotopy$ on $\totallyadmissiblesubset$ is in $\mbarthreetop_\tree$.
\end{prop}
The proof of this proposition is elementary, albeit intricate, and will be deferred to Appendix~\ref{appendix:homotopy}.
\begin{proof}[of Lemma~\ref{lemma:giveswe1}]
We have constructed a section, $\secty$.  Because this is a section, for $(\psplone_\iti)$ in the open stratum of $\mbarthree_\tree$, $(\psplone_\iti,(\secty\longmap \psplone)_\iti)$ is in $\totallyadmissiblesubset$ so that $(\id,\secty\longmap)$ is a map from $\mbarthree_\tree$ to $\totallyadmissiblesubset$.  Then $\homotopy\circ (\id,\secty\longmap)$ is a homotopy between $\id$ and $\secty\longmap$ on the open stratum of $\mbarthree_\tree$ over $\nei_\tree$.

It only remains to show that this homotopy has the appropriate limiting behavior as one approaches a lower stratum.  This follows from the facts that $\secty$ has the appropriate behavior as described above and that $\homotopypart$ is $\affc$-equivariant, with the simultaneous action on both factors of $\overlappedset$.
\end{proof}
\begin{lemma}\label{lemma:giveswe2}
Let $\Unei$ be the intersection of a finite set $\{\nei_{\tree_\iti}\}$.  Then for all $\iti$, the map from the preimage of $\unei$ in $\mbarthree_{\tree_\iti}$ to the preimage of $\unei$ in $\mbarone$ is a weak equivalence.
\end{lemma}
\begin{proof}
We will eventually follow the proof of Lemma~\ref{lemma:giveswe1}, but cannot do so immediately because it is not obvious how to pick a section without modifying the situation.  This is because of the equivalence relation that equates different configurations with ``annuli'' that have inner radius one in $\mbarone$.  

First, the preimage of $\unei$ in $\mbarone$ may contain configurations with annuli that don't correspond to vertices of $\tree_\iti$.  There is a retraction onto the subspace containing no such inappropriate annuli.  None of these illegal configurations are in the image of $\mbarthree_{\tree_\iti}$ so we can begin by replacing the preimage with its retract.   Let $\someedges_1,\someedges_2$ be sets of edges of $\tree_\iti$ so that $\someedges_1\cup\someedges_2=\edgesof(\tree_\iti)$.  Let $\verticesof_{\someedges_1,\someedges_2}\subset \mbarone$ be the the subset of the preimage of $\unei$ so that the inner radius of every annulus corresponding to an edge in $\someedges_1$ is greater than zero and the inner radius of every annulus corresponding to an edge in $\someedges_2$ is less than one.  These are open sets and cover the preimage of $\unei$, as every radius is either greater than zero or less than one.  They are closed under intersection, as $\verticesof_{\someedges_1,\someedges_2}\cap \verticesof_{\someedges'_1,\someedges'_2}=\verticesof_{\someedges_1\cup\someedges'_1,\someedges_2\cup\someedges'_2}$.  So by Lemma~\ref{lemma:folklore},  it suffices to show that the map from the preimage in $\mbarthree_{\tree_\iti}$ of $\verticesof_{\someedges_1,\someedges_2}$ is a weak equivalence.

In both the domain and the range there is a deformation retraction to the subset where the radius of every annulus corresponding to an edge in $\someedges_1\backslash\someedges_2$ has radius one.  This involves making a choice; we shall hold the outer radius fixed and increase the inner radius, modifying the little disk inside it.  Now it suffices to show that the map in question is a weak equivalence when restricted to these deformation retracts.  

In this restricted situation, we have specified center and inner and outer radii for each of the edges in $\someedges_2$, so to specify a section, it suffices to choose a center and outer radius for each edge in $\someedges_1\backslash \someedges_2$ (the inner radius must be one).

We will use the centers of mass $\cog_\edge$ as centers.  We can construct $\rmin$ and $\rmax$ as in the construction of $\secty$ for the edges in $\someedges_1\backslash \someedges_2$, using the fixed inner radii of the other annuli in place of $\rmin$ and $\rsmall$ (respectively the outer radii for $\rbig$).  We choose only one of the radii, say, $\rsmall$, for the edges corresponding to annuli of inner radius one, using it as both $\rsmall$ and $\rbig$.

The homotopy used in the proof of Lemma~\ref{lemma:giveswe1} can be used without modification because it preserves annuli of inner radius one, annuli with inner radius less than one, and annuli with inner radius greater than zero.
\end{proof}
\begin{proof}[of Theorem~\ref{thm:toptheoremtwo}] 
Let $\unei$ be a finite intersection of some collection of $\tree$-neighborhoods in $\Mbar$.  By Lemmas~\ref{lemma:DRoY}~and~\ref{lemma:giveswe1}, the maps from the preimages of $\unei$ in $\mbarthree_{\tree_\iti}$ to $\unei$ are weak equivalences. By Lemma~\ref{lemma:giveswe2} and the two-out-of three axiom, the maps from the preimages of $\unei$ in $\mbarone$ to $\unei$ are weak equivalences.

Neighborhoods of the form $\unei$ form a cover of $\Mbar$ with the finite intersection property. Thus by Lemma~\ref{lemma:folklore}, $\mbarone\to\Mbar$ is a weak equivalence.
\end{proof}
%
%
%
\appendix
\section{Model category background}\label{sec:models}
We briefly describe model categories, which are used to define the homotopy pushout. \cite{Hovey:MC}~and~\cite{Hirschhorn:MCL} are good general references; we will not state all of the background to the theory but instead will use several results as a black box.
\begin{defi}
Let $\morphismone$ and $\morphismtwo$ be morphisms in a category.  We say that $\morphismone$ has the left lifting property with respect to $\morphismtwo$ and that $\morphismtwo$ has the right lifting property with respect to $\morphismone$ if there is a dotted filler morphism for every solid diagram of the form
\[
  \xymatrix{
    \firstobject\ar[d]_\morphismone\ar[r]&\secondobject\ar[d]^\morphismtwo\\\thirdobject\ar@{.>}[ur]\ar[r]&\fourthobject
    }\]
    \end{defi}
\begin{defi}
A {\em model category} is a bicomplete category equipped with three subcategories, called {\em fibrations}, {\em cofibrations}, and {\em weak equivalences} satisfying the axioms below.  A morphism that is both a weak equivalence and in one of the other two subcategories is called {\em trivial}. An object in the model category is called {\em cofibrant} if the morphism from the initial object to it is a cofibration and {\em fibrant} if the morphism to the terminal object is a fibration.  The axioms are:
\begin{enumerate}
\item all classes are closed under retracts,
\item (two of three) if $\morphismone$ and $\morphismtwo$ are composable morphisms so that two of $\morphismone$, $\morphismtwo$, and $\morphismone\circ\morphismtwo$ are weak equivalences, so is the third.
\item A morphism is a (trivial) cofibration if and only if it has the left lifting property with respect to all trivial fibrations (fibrations), and a morphism is a (trivial) fibration if and only if it has the right lifting property with respoct to all trivial cofibrations (cofibrations).  Thus the weak equivalences and the fibrations determine the cofibrations, among other similar statements.
\item Every morphism can be factored as a cofibration followed by a fibration, either one of which may be chosen to be a weak equivalence.
\end{enumerate}
\end{defi}
We will need the following facts:
\begin{prop}\label{prop:modelcatfact}
\begin{enumerate}
\item If
\[
  \xymatrix{
    \firstobject\ar[d]_\morphismone\ar[r]&\secondobject\ar[d]^\morphismtwo\\\thirdobject\ar[r]&\fourthobject
    }\]
is a pushout diagram and $\morphismone$ is a cofibration, so is $\morphismtwo$
\item The category of (weak Hausdorff) compactly generated spaces is a model category where the fibrations are the Serre fibrations and the weak equivalences are maps inducing isomorphisms on all homotopy groups.
\item Let $\spacea\to \spaceb$ and $\spacec\to \spaced$ be cofibrations of spaces.  Then $(\spacea\times \spaced)\sqcup_{\spacea\times \spacec}(\spaceb\times \spacec)\to \spaceb\times \spaced$ is a cofibration.
\end{enumerate}
\end{prop}
This last has two useful corollaries:
\begin{cor}\label{cor:prodcofibration}
\begin{enumerate}\item Let $\spacea\to \spaceb$ be a cofibration of spaces and $\spaced$ a cofibrant space.  Then $\spacea\times \spaced\to \spaceb\times \spaced$ is a cofibration of spaces.
\item
Let $\spaceb$ and $\spaced$ be cofibrant spaces.  Then $\spaceb\times\spaced$ is cofibrant.
\end{enumerate}
\end{cor}

We will use several special Reedy categories in passing.  These are diagram categories which are well-behaved with respect to model categories.  More details about this tool can be found in Chapter~15 of~\cite{Hirschhorn:MCL}
\begin{defi}\label{defi:direct}
Let $\nN$ denote the category whose objects are natural numbers and where there is a unique morphism $\itm\to\itn$ if $\itm\le\itn$. 
A (small) {\em direct category} is a small category equipped with a functor to $\nN$ that takes nonidentity morphisms to nonidentity morphisms.  
\end{defi}
\begin{defi}
A {\em nearly direct category} is a small category $\category$ equipped with a direct subcategory so that the complement of the direct subcategory contains at most one morphism and is closed under composition.  For example, any direct category is nearly direct.
\end{defi}
\begin{prop}\label{prop:reedyprop}[\cite{Hirschhorn:MCL}, Theorems 15.3.4 and 15.10.9]
Let $\category$ be a model category and $\diagram$ a nearly direct category.  Then the category $\category^\diagram$ of $\diagram$-diagrams in $\category$ has the structure of a model category where the weak equivalences are objectwise weak equivalences of diagrams and cofibrant objects are diagrams with every object cofibrant in $\category$ and every morphism in the direct subcategory a cofibration.  If $\diagram$ is a direct category, then the fibrations are objectwise fibrations.

In either case, the colimit functor $\category^\diagram\to\category$ takes cofibrations and weak equivalences between cofibrant objects to cofibrations and weak equivalences in $\category$.
\end{prop}
\begin{lemma}\label{lemma:cubereedy}
The following are direct categories:
\begin{enumerate}
\item The deleted $\finiteset$-cube category $\Scube$ (see Definition~\ref{defi:scube})
\item The pushout category $\bullet\gets\bullet\to\bullet$ (this is the deleted $\{0,1\}$-cube category).
\item The telescope category $\bullet\to\bullet\to\bullet\to\cdots$
\end{enumerate}
The pushout category is also a nearly direct category, with direct subcategory missing one of the two non-identity morphisms.
\end{lemma}
\begin{defi}
Let $\diagram$ be the pushout category, considered as a nearly direct category.  Let $\firstobject$ be a pushout diagram in a model category $\category$.
\begin{enumerate}
\item A {\em model of $\firstobject$} is a diagram $\replacement{\firstobject}$ equipped with a morphism $\replacement{\firstobject}\to\firstobject$ which is a weak equivalence in $\category^\diagram$ (an objectwise weak equivalence).  Its {\em realization} is its colimit.
\item A {\em cofibrant model} of $\firstobject$ is a model of $\firstobject$ that is cofibrant in $\category^\diagram$. 
\item A {\em homotopically-correct model} of $\firstobject$ is a model of $\firstobject$ which accepts a weak equivalence from a cofibrant model of $\firstobject$ that passes to a weak equivalence of realizations. The realization of a homotopically correct model of $\firstobject$ is a {\em realization of the homotopy pushout} of $\firstobject$.
\item A {\em model of the homotopy pushout} of $\firstobject$ is an object of the category $\category$ equipped with a weak equivalence to a realization of the homotopy pushout of $\firstobject$.
\end{enumerate}
\end{defi}
\begin{lemma}
Let $\firstobject\to\firstobject'$ be a weak equivalence of pushout diagrams in a model category.  
\begin{enumerate}
\item There is a zigzag of weak equivalences over $\firstobject$ between two cofibrant models of $\firstobject$ inducing weak equivalences of their realizations.
\item There is a zigzag of weak equivalences over $\firstobject'$ between any cofibrant model of $\firstobject$ and any cofibrant model of $\firstobject'$.
\item There is a zigzag of weak equivalences between any pair of models of the homotopy pushouts of $\firstobject$ and $\firstobject'$.
\end{enumerate}
\end{lemma}
\begin{proof}
\begin{enumerate}
\item Let $\widetilde{\firstobject}$ and $\widetilde{\firstobject}'$ be two cofibrant models of $\firstobject$.  Factorize $\widetilde{\firstobject}\to\firstobject$ as a trivial cofibration followed by a trivial fibration.  The intermediate diagram is cofibrant and there is a lift from $\widetilde{\firstobject}'$ into it over $\firstobject$.  Everything is a weak equivalence by the two-of-three axiom and the weak equivalences pass to realizations by Proposition~\ref{prop:reedyprop}.
\item Let $\widetilde{\firstobject}$ and $\widetilde{\firstobject}'$ be the two cofibrant models.  Then $\widetilde{\firstobject}$ is also a cofibrant model of $\firstobject'$.  Factorize as before.
\item This follows from the first parts of the lemma and the definition.
\end{enumerate}
\end{proof}
%
%
%
\section{Trees}\label{appendix:trees}
We fix the conventions used for trees.
\begin{defi}
Let $\finiteset$ be a nonempty finite set. A {\em non-planar unreduced rooted $\finiteset$-tree} $\tree$ (or, for notational ease, just a {\em tree}) consists of the following data:
\begin{enumerate}
\item a finite set $\verticesof(\tree)$ of {\em vertices} and
\item a finite set $\edgesof(\tree)$ of {\em edges} the elements of which are sets of two elements in $\verticesof(\tree)\sqcup\finiteset$ along with a single special edge, called the {\em root edge}, which is a singleton from $\verticesof(\tree)\sqcup\finiteset$.
\end{enumerate}
The set $\finiteset$ will be called the set of {\em leaves} of $\tree$. The set of {\em nodes} denotes the disjoint union of the vertices and the leaves. The edges other than the root edge are {\em proper}.

This data should satisfy the following conditions.
\begin{enumerate}
\item (connectedness and simply-connectedness) Given any pair $\element$ and $\element'$ of distinct nodes, there is a unique sequence of distinct proper edges $\edge_0,\ldots,\edge_\itn$ so that $\element$ is a member of $\edge_0$, $\element'$ is an element of $\edge_\itn$, and $\edge_{i-1}\cap \edge_{i}$ is nonempty for $i\in\{1,\ldots, n\}$. Such a sequence is called a {\em path} from $\element$ to $\element'$.
\item The leaves are elements of precisely one edge each and each vertex is an element of at least two edges.
\end{enumerate}
Each node $\element$ other than the root has a unique {\em output edge} which is the first edge $\edge_0$ in the path from $\element$ to the root. The {\em successor vertex} to $\element$ is the other element of the output edge of $\element$. Each proper edge $\{\element, \element'\}$ is the output edge of one of its constituent nodes; we call that node the {\em input} of the edge and the other node the {\em output} of the edge. By convention, the root edge is the output edge of the root. There is a partial order on nodes and/or edges generated by the successor relation. So the node $\element$ is {\em above} the node ($\element'$ or the edge $\edge$) if its path to the root passes through $\element'$ ($\edge$). Similarly, an edge is above its output vertex and all edges and nodes below that output vertex.

Edges which contain a leaf and the root edge are called {\em external} edges and all other edges are called {\em internal} edges. A tree with one vertex is called a {\em corolla}. A vertex which is an element of precisely two edges is called {\em bivalent}.

Let $\extedges(\tree)$ and $\intedges(\tree)$ denote the external and internal edges of $\tree$, and let $\edgesof(\vertex)$ (similarly $\extedges(\vertex), \intedges(\vertex)$) denote the edges which have $\vertex$ as output.  
\end{defi}
\begin{defi}
A tree is {\em stable} (also called {\em reduced} in the literature) if it has no bivalent vertices.
\\
A tree is {\em nearly stable} if each bivalent vertex either is itself the root of the tree or shares an edge with a leaf.
\end{defi}
As defined, there is no set of $\finiteset$-trees, so we pass to the set of isomorphism classes of trees. Since $\finiteset$-trees do not have automorphisms, this causes no difficulties. Choose a representative of each isomorphism class once and for all. Any tree we create will be canonically isomorphic to one such representative, and we suppress all such isomorphisms and set theoretical issues.

Now call the set of (representative) stable $\finiteset$-trees $\treesone{\finiteset}$ and of stable $\finiteset$-trees with $\itb$ vertices $\trees{\finiteset}{\itb}$.

\begin{defi}
Let $\tree$ be a tree and $\edge$ an edge.  Create a new tree whose leaves are the same as the leaves of $\tree$ and whose vertices are the vertices of $\tree$, except adjoin a new vertex $\element'$. The edges of the new tree are the edges of $\tree$ except we remove the chosen edge $\edge$ and replace it with two edges. If the original edge was proper, say, $\{\element,\element''\}$, then we replace it with $\{\element,\element'\}$ and $\{\element',\element''\}$. If the original edge was the root edge $\{\element\}$ then we replace it with $\{\element,\element'\}$ and $\{\element'\}$. This new tree is the tree obtained from $\tree$ by {\em inserting a vertex} on $\edge$. 
\end{defi}
\begin{defi}
Let $\tree$ be a tree and $\edge=\{\vertex,\vertextwo\}$ an internal edge of $\tree$.  The tree $\tree_\edge$ obtained by {\em contracting} the edge $\edge$ is as follows. Its leaves are the same as those of $\tree$, and its vertices are the quotient of the vertex set of $\tree$ by the equivalence relation $\vertex=\vertextwo$. The edges of $\tree_\edge$ are the edges of $\tree$, except for $\edge$ itself (called the {\em contraction edge}), subject to the equivalence relation on vertices

The quotient vertex $\{\vertex,\vertextwo\}$ is called the {\em contraction vertex}. Notationally this is the same as an edge, but there will be no confusion in practice.  We consider it to be the identification of the output and input of $\edge$.  

If a vertex $\vertex$ is an element of precisely two edges, then contracting the output edge of $\vertex$ is also called {\em forgetting the vertex} $\vertex$.

If $\someedges$ is a set of internal edges, then $\tree_\someedges$ is the tree obtained by contracting all the edges in $\someedges$.  The order of contraction does not matter.

The set $\treesone{\finiteset}$ form the objects of a (small) category where the morphisms out of a tree $\tree$ are indexed by sets of internal edges: $\someedges$ is a morphism from $\tree$ to $\tree_\someedges$.  Composition uses the identification above to define a bigger contraction set. 
\end{defi}
\begin{defi}
Let $\tree$ and $\tree'$ be $\finiteset$ and $\finiteset'$-trees.  Let $\leaf$ be a leaf of $\tree$, and let $\mapofsets:\finiteset'\sqcup \finiteset\backslash\{\leaf\}\to \finiteset''$ be a relabelling isomorphism.  Then the $\finiteset''$-tree obtained by {\em grafting} $\tree'$ to $\tree$ at the leaf $\leaf$ is as follows.
\begin{enumerate}
\item The vertex set is equal to the disjoint union of the vertices of $\tree$ and the vertices of $\tree'$.
\item The edge set is equal to the disjoint union of the edges of $\tree$ and $\tree'$ without the unique edge containing $\leaf$ and the root edge of $\tree'$, but with a new {\em grafting edge}. This new edge has as its elements the root of $\tree'$ and the successor of $\leaf$ in $\tree$ if it exists (if $\leaf$ is the root of $\tree$ then it will not have a successor).
\end{enumerate}
\end{defi}
%
%
\section{Proof of Proposition~\ref{prop:homotopyconstruction}}\label{appendix:homotopy} 
Proposition~\ref{prop:homotopyconstruction} says that the homotopy $\homotopy$ lands in $\mbarthreetop_\tree$.  The content of this proposition is that the homotopy preserves the configuration in which disks are nested in one another.  That is, if $\diskone_0$ and $\diskone_1$ are properly overlapping disks (and likewise $\disktwo_0$ and $\disktwo_1$), then:
  \begin{enumerate}
  \item If $\diskone_\iti$ is contained in $\disktwo_\iti$ for $\iti\in\{0,1\}$, then for all $\timez$, the disk which is the image of the pair $\diskone_\iti$ under the homotopy slice $\homotopy_\timez$ is contained in the disk which is the image of the pair $\disktwo_\iti$ under the homotopy slice $\homotopy_\timez$.
  \item If $\diskone_\iti$ and $\disktwo_\iti$ are disjoint for $\iti\in\{0,1\}$, then for all $\timez$, the disks which are the images under the homotopy slice $\homotopy_\timez$ of the pairs $\diskone_\iti$ and $\disktwo_\iti$ are disjoint.
  \end{enumerate}
These two statements are, respectively, Corollaries~\ref{cor:fullnested}~and~\ref{cor:fullvenn}, and by proving them, we prove the proposition.

\begin{lemma}\label{lemma:randccalculations}
The pair $\homotopypart_\timez(\cent_0,\radi_0,\cent_1,\radi_1,\iti)$ coincides with $(\cent_\iti, \radi_\iti)$ for $\iti=0,1$, justifying the notation $(\cent_\timez,\radi_\timez)$. The disk centered at $\cent_\timez$ of radius $\radi_\timez$ contains the intersection of the disks centered at $\cent_\iti$ of radius $\radi_\iti$, for $\iti\in 0,1$.
\end{lemma}
\begin{proof}
The first statement follows immediately from the definition.
Let $\varone$ be a point in the intersection.  We will prove the second statement in the case that $\cent_0$, $\cent_1$, and $\varone$ are disjoint; the other cases are easier.  In this case, there is a (possibly degenerate) triangle with corners $\cent_0$, $\cent_1$, and $\varone$.  The point $\cent_\timez$ is on the segment between $\cent_0$ and $\cent_1$.  Repeated applications of the law of cosines and substitution show that the distance from $\cent_\timez$ to $\varone$ satisfies
\[
|\cent_\timez-\varone|^2=\frac{|\cent_0-\varone|^2|\cent_1-\cent_\timez|+|\cent_1-\varone|^2|\cent_0-\cent_\timez|}{|\cent_0-\cent_1|}-|\cent_0-\cent_\timez||\cent_1-\cent_\timez|.
\]
Using the facts that $|\cent_0-\varone|\le \radi_0$ and $|\cent_1-\varone|\le \radi_1$, the definitions of $\angleof$ and $\cent_\timez$, and trigometric identities, we can use this to generate the inequality
\[
|\cent_\timez-\varone|^2\le \radi_\timez^2\left(2\cos\angleof \sin(\timez\angleof)\sin((1-\timez)\angleof)+\sin^2(\timez\angleof)+\sin^2((1-\timez)\angleof)\right)=\radi_\timez^2\sin^2\angleof,\]
which is less than or equal to $\radi_\timez^2$.
\end{proof}
\begin{lemma}\label{lemma:halves}Let $|\cent_0-\cent_1|<\radi_0+\radi_1$.  Then: 
\begin{enumerate}
\item $\cent_\frac{1}{2}=\frac{\cent_1\radi_0+\cent_0\radi_1}{\radi_0+\radi_1}$,
\item $\radi_\frac{1}{2}=\frac{2\radi_0\radi_1\cos(\frac{\angleof}{2})}{\radi_0+\radi_1}=\frac{\sqrt{\radi_0\radi_1}\sqrt{(\radi_0+\radi_1)^2-|\cent_0-\cent_1|^2}}{\radi_0+\radi_1}$, 
\item $\angleof(\cent_0,\cent_\frac{1}{2},\radi_0,\radi_\frac{1}{2})=\frac{\angleof}{2}=\angleof(\cent_\frac{1}{2},\cent_1,\radi_\frac{1}{2},\radi_1)$, and
\item $\homotopypart(\cent_0,\cent_\frac{1}{2},\radi_0,\radi_\frac{1}{2},2\timez)=\homotopypart(\cent_0,\cent_1,\radi_0,\radi_1,\timez)=
\homotopypart(\cent_\frac{1}{2},\cent_1,\radi_\frac{1}{2},\radi_1,2\timez-1).$
\end{enumerate}
\end{lemma}
These are all straightforward evaluations, with some trigonometric substitutions.
\begin{defi}Let $\diskone$ and $\disktwo$ be disks.  The {\em nesting distance} $\nest{\diskone}{\disktwo}$ is $\radi_\disktwo-\radi_\diskone-|\cent_\diskone-\cent_\disktwo|$.  This quantity is nonnegative (positive) if $\diskone$ is (properly) nested in $\disktwo$; in this case it is the distance from the interior of $\diskone$ to the exterior of $\disktwo$. 
\end{defi}
\begin{remark*}
By the triangle inequality the nesting distance is superadditive; that is, $\nest{\diskone}{\diskthree}\ge \nest{\diskone}{\disktwo}+\nest{\disktwo}{\diskthree}$.
\end{remark*}
\begin{notation}
If $\diskone_0=(\cent_0,\radi_0)$ and $\diskone_1=(\cent_1,\radi_1)$ are properly overlapping disks, then we will use $\diskone_\timez$ to refer to the disk $(\cent_\timez,\radi_\timez)$ determined from these two by $\homotopypart$.
\end{notation}
\begin{prop}\label{prop:spacing}
Let $\{\diskone_0,\diskone_1\}$ and $\{\disktwo_0,\disktwo_1\}$ be two pairs of properly overlapping disks.  Suppose $\diskone_\iti$ is nested in $\disktwo_\iti$ for $\iti\in\{0,1\}$.  Then $\diskone_{\frac{1}{2}}$ is nested in $\disktwo_{\frac{1}{2}}$, and, in fact, $\nest{\diskone_{\frac{1}{2}}}{\disktwo_{\frac{1}{2}}}$ is greater than or equal to the minimum of $\nest{\diskone_0}{\disktwo_0}$ and $\nest{\diskone_1}{\disktwo_1}$.  Furthermore, $\nest{\diskone_{\frac{1}{2}}}{\disktwo_{\frac{1}{2}}}$ can only be zero if {\em both} $\nest{\diskone_0}{\disktwo_0}$ and $\nest{\diskone_1}{\disktwo_1}$ are zero.
\end{prop}
\begin{lemma}\label{lemma:spacingcent}
Proposition~\ref{prop:spacing} is true in the special case where $\diskone_0$ and $\disktwo_0$ have the same center, $\cent_0$, and $\diskone_1$ and $\disktwo_1$ have the same center, $\cent_1$.
\end{lemma}
\begin{proof}
In this case we can write 
\[
\cent(\diskone_0,\diskone_1,\frac{1}{2})-\cent(\disktwo_0,\disktwo_1,\frac{1}{2})=
\frac{(\cent_1-\cent_0) (\radi_{\diskone_0}\radi_{\disktwo_1}-\radi_{\diskone_1}\radi_{\disktwo_0})}{(\radi_{\diskone_0}+\radi_{\diskone_1})(\radi_{\disktwo_0}+\radi_{\disktwo_1})}\]
Reparameterize using the variables 
\begin{eqnarray*}
\vara&=&|\cent_0-\cent_1|\\
\varb&=&\radi_{\diskone_0}+\radi_{\diskone_1}\\
\varc&=&\radi_{\diskone_0}-\radi_{\diskone_1}\\
\vard&=&\radi_{\disktwo_0}+\radi_{\disktwo_1}\\
\vare&=&\radi_{\disktwo_0}-\radi_{\disktwo_1}.
\end{eqnarray*}
Note that $\varb$ and $\vard$ are always positive, $\vard\ge \varb>\vara$, that $-\varb<\varc<\varb$, and that $-\vard<\vare<\vard$.  In fact, $\varb+\varc-\vard\le \vare\le -\varb+\varc+\vard$.  Then we can define $\dummyfunction(\vara,\varb,\varc,\vard,\vare)=\nesthalf-\nestsbase$
and write \[
\dummyfunction=
\frac{\sqrt{\vard^2-\vare^2}\sqrt{\vard^2-\vara^2}}{2\vard}-
\frac{\sqrt{\varb^2-\varc^2}\sqrt{\varb^2-\vara^2}}{2\varb}-
\frac{\vara|\varc\vard-\varb\vare|}{2\varb\vard}+\frac{\varb-\vard}{2}+\frac{|\varc-\vare|}{2}.
\]
This is piecewise differentiable with respect to $\vare$ with derivative
\[
\dummyfunction_\vare=-\frac{\vare\sqrt{\vard^2-\vara^2}}{2\vard\sqrt{\vard^2-\vare^2}}\pm \frac{\vara}{2\vard}\pm \frac{1}{2}
\]
This is really four different functions, depending on the signs of $\frac{u}{2y}$ and $\frac{1}{2}$.   Each of the four functions is defined globally. For each of them, as $\vare$ approaches $-\vard$ $\dummyfunction_\vare$ is positive and as $\vare$ approaches $\vard$, $\dummyfunction_\vare$ is negative.  Further, each of these four globally defined functions has a single zero (at $\vare=\pm \sqrt{\frac{\vard(\vard\pm \vara)}{2}}$), which may or may not be in the appropriate domain for that choice of signs.    This means that any zero of the overall piecewise function $\dummyfunction_\vare$ is a local maximum, so local minima can only occur at boundary points and at the discontinuities of the derivative, which occur at $\vare=\varc$ and $\vare=\frac{\vard}{\varb}\varc$.  

At the discontinuity $\vare=\frac{\vard}{\varb}\varc$, say that the function changes from $\dummyfunction_{\vare,1}$ to $\dummyfunction_{\vare,2}$.  If $\varc\ne 0$, then the zero of $\dummyfunction_{\vare,1}$ is at least as great as the zero of $\dummyfunction_{\vare,2}$.  So either the global function $\dummyfunction_\vare$ increases to the left of $\vare=\frac{\vard}{\varb}\varc$, decreases to the right of it, or both, and in particular, this cannot be a minimum of $\dummyfunction$.  If $\varc=0$ then $\frac{\vard}{\varb}\varc$ coincides with $\varc$.

To summarize the progress at this stage of the proof, we have shown that holding the first four variables constant, $\dummyfunction$ achieves its minimal value for $\vare$ in the interval $[\varb+\varc-\vard,-\varb+\varc+\vard]$ at either one of the endpoints or the midpoint $\vare=\varc$.  We will show that in these minimal cases, $\dummyfunction\ge 0$.  

At the endpoints $\vare=\varc\pm(\varb-\vard)$, the function $\dummyfunction$ is continuously differentiable in $\vard$, with positive derivative.  

So for the specified endpoint values of $\vare$, $\dummyfunction$ increases as $\vard$ increases.  The minimal value of $\vard$ is at least $\varb$, and we have $\dummyfunction(\vara,\varb,\varc,\varb,\varc)=0$.

For the other critical point, $\vare=\varc$, the derivative with respect to $\vard$ is generically positive, implying that the minimal value of $\dummyfunction$ is at least zero, except when $\vare=\varc$.  In this case, direct inspection shows $\dummyfunction$ to be identitically zero.  

We have shown the desired inequality; it remains to show that $\nest{\diskone_{\frac{1}{2}}}{\disktwo_{\frac{1}{2}}}$ can only be zero if {\em both} $\nest{\diskone_0}{\disktwo_0}$ and $\nest{\diskone_1}{\disktwo_1}$ are zero.  By the definition of $\dummyfunction$, this can only happen if $\dummyfunction$ is zero and one of the two quantities $\nest{\diskone_0}{\disktwo_0}$ and $\nest{\diskone_1}{\disktwo_1}$ is zero.  This can only occur in two cases.  In the first case, $\varb=\vard$ and $\varc=\vare$.  In this case $\diskone_\iti$ coincides with $\disktwo_\iti$ for both $\iti=0$ and $\iti=1$.  In the second case, $\varc=\vare$ and $\vara=|\varc|$.  In this case, without loss of generality, if $\nest{\diskone_0}{\disktwo_0}$ is zero then $\radi_{\diskone_0}=\radi_{\disktwo_0}$ and then $\vare=\varc$ implies $\radi_{\diskone_1}=\radi_{\disktwo_1}$ so $\nest{\diskone_1}{\disktwo_1}=0$ as well.
\end{proof}
\begin{lemma}\label{lemma:spacingtang}
Proposition~\ref{prop:spacing} is true in the special case where $\diskone_0$ and $\disktwo_0$ are tangent and $\diskone_1$ and $\disktwo_1$ coincide.  The proposition is also true in the special case where $\diskone_0$ and $\disktwo_0$ coincide and $\diskone_1$ and $\disktwo_1$ are tangent.
\end{lemma}
\begin{proof}
We prove only the first statement; the second statement's proof is essentially identical.  If the center $\centertwo$ of $\disktwo_0$ coincides with either the center $\centerone$ of $\diskone_0$ or the center $\centerthree$ of $\diskone_1$ and $\disktwo_1$, then the proof can be completed by simple algebraic manipulation.  If $\centertwo$ is distinct from both $\centerone$ and $\centerthree$, we can parameterize with:
\begin{eqnarray*}
\vara&=&|\centertwo-\centerthree|\\
\varb&=&\radi_{\diskone_0}+\radi_{\diskone_1}\\
\varc&=&\radi_{\diskone_1}=\radi_{\disktwo_1}\\
\vard&=&\radi_{\disktwo_0}+\radi_{\disktwo_1}\\
\vare&=&\cos\angleof,\end{eqnarray*}
where $\angleof$ is the angle formed at $\centertwo$ by the rays toward $\centerone$ and $\centerthree$. Note that $|\centerone-\centertwo|=\vard-\varb$ and then by the law of cosines, \[|\centerone-\centerthree|=\sqrt{\vara^2+(\vard-\varb)^2-2\vara(\vard-\varb)\vare}\] and $|\cent_{\diskone_{\frac{1}{2}}}-\cent_{\disktwo_{\frac{1}{2}}}|$ is $\frac{\varc(\vard-\varb)\sqrt{\vara^2+\vard^2-2\vara\vard\vare}}{\varb\vard}$.  These formulae are still valid if $\centertwo$ coincides with $\centerone$ or $\centerthree$, even though $\vare$ is no longer well-defined in these cases.

Now, defining $\dummyfunction(\vara,\varb,\varc,\vard,\vare)$ as in the previous lemma with these different variables, we can differentiate with respect to $\vare$, and find that $\dummyfunction$ has a unique global minimum at 
\[
\vare_0=\frac{\vard^2+\vara^2-2\varc\vard}{2\vara(\vard-\varc)},
\]
which may or may not be in $[-1,1]$.  Whether or not it is, $\dummyfunction(\vare)\ge \dummyfunction(\vare_0)=0$, with equality only at $\vare=\vare_0$.

\end{proof}
\begin{proof}[of Proposition~\ref{prop:spacing}]
Let $\{\diskone_0,\diskone_1,\disktwo_0,\disktwo_1\}$ be disks as in the hypotheses of the proposition.  Then $\diskone_0$ is contained in a disk $\diskone_0'$ with the same center which is tangent to and contained in $\disktwo_0$, and likewise $\diskone_1$ is contained in a disk $\diskone_1'$ with the same center which is tangent to and contained in $\disktwo_1$.  Then by Lemma~\ref{lemma:spacingcent}, $\nest{\diskone_{\frac{1}{2}}}{\diskone_{\frac{1}{2}}'}\ge \min\{\nest{\diskone_0}{\diskone_0'},\nest{\diskone_1}{\diskone_1'}\}$.  If $\diskthree_0=\disktwo_0$ and $\diskthree_1=\diskone_1'$, then Lemma~\ref{lemma:spacingtang} shows that $\diskone_{\frac{1}{2}}'$ is nested in $\diskthree_{\frac{1}{2}}$ and that $\diskthree_{\frac{1}{2}}$ is nested in $\disktwo_{\frac{1}{2}}$.  So 
\[\nest{\diskone_{\frac{1}{2}}}{\disktwo_{\frac{1}{2}}}\ge \nest{\diskone_{\frac{1}{2}}}{\diskone_{\frac{1}{2}}'}\\\ge \min\{\nest{\diskone_0}{\diskone_0'},\nest{\diskone_1}{\diskone_1'}\} \ge \min\{\nest{\diskone_0}{\disktwo_0},\nest{\diskone_1}{\disktwo_1}\}.\]
We have shown that $\nest{\diskone_\frac{1}{2}}{\diskone_\frac{1}{2}'}$ is greater than zero unless $\diskone_\iti$ and $\diskone_\iti'$ coincide for both $=\iti=0$ and $\iti=1$, in which case $\diskone_\iti$ and $\disktwo_\iti$ coincide or are tangent for both $\iti=0$ and $\iti=1$.
\end{proof}
\begin{cor}\label{cor:fullnested}
Let $\{\diskone_0,\diskone_1,\disktwo_0,\disktwo_1\}$ be as in Proposition~\ref{prop:spacing}.  Then $\diskone_\timez$ is nested in $\disktwo_\timez$ for all $\timez\in [0,1]$, that is, $\nest{\diskone_\timez}{\disktwo_\timez}\ge 0$.  This inequality is strict for $\timez\in (0,1)$ unless $\nest{\diskone_0}{\disktwo_0}=\nest{\diskone_1}{\disktwo_1}=0$.
\end{cor}
\begin{proof}
We begin by showing the inequality for rational $\timez$ of the form $\frac{\integerthing}{2^\integerpower}$ ({\em dyadic rationals}).  We proceed by induction on $\integerpower$.  If $\integerpower=0$, the statement is true by assumption.  Suppose we have shown the statement for $\integerpower<\fixedpower$.  By repeated applications of Lemma~\ref{lemma:halves} and Proposition~\ref{prop:spacing}, the statement is true for $\integerpower=\fixedpower$.   Since being nested is a closed condition, being true on a dense subset of the interval implies that it is true throughout the interval.  

Now assume, without loss of generality, that $\nest{\diskone_0}{\disktwo_0}>0$.  By the same induction, $\nest{\diskone_\timez}{\disktwo_\timez}>0$ for dyadic rationals except possibly at $\timez=1$.  Then for irrational $\timez$, pick two consecutive such rationals bracketing it: $\lowerrational=\frac{\integerthing}{2^\integerpower}<\timez<\higherrational=\frac{\integerthing+1}{2^\integerpower}<1$.  Then as the induction proceeds, for all dyadic rationals $\middlerational$ between $\lowerrational$ and $\higherrational$, $\nest{\diskone_\middlerational}{\disktwo_\middlerational}\ge \min\{\nest{\diskone_\lowerrational}{\disktwo_\lowerrational},\nest{\diskone_\higherrational}{\disktwo_\higherrational}\}>0$.  Then this inequality holds in the limit and $\nest{\diskone_\timez}{\disktwo_\timez}$ is bounded away from zero.
\end{proof}
\begin{defi}
Let $\halfplanes$ be the space of half-planes in the standard plane, identified with $\thecircle\times \rR$ in the following manner: the pair $(\angleof, \displacement)$ consists of points in the plane whose inner product with $(\cos\angleof,\sin\angleof)$ is greater than (or greater than or equal to) $\displacement$.  We will not distinguish between open and closed half-planes, which are in one-to-one correspondence.  Let the space of {\em admissible pairs of half-planes}, $\admissibles\subset \halfplanes^2$ be the subspace of points $(\angleof_1,\displacement_1,\angleof_2,\displacement_2)$ where $\angleof_1\ne -\angleof_2$.  

We define the {\em intermediate half plane} of an admissible pair as a map $\admissibles\to\halfplanes$ defined as
\[
  (\angleof_1,\displacement_1,\angleof_2,\displacement_2)\mapsto \left(\frac{\angleof_1+\angleof_2}{2},\frac{\displacement_1+\displacement_2}{2\cos\frac{\angleof_1-\angleof_2}{2}}\right)\]
where representatives of $\angleof_1$ and $\angleof_2$ are chosen so that they differ by less than $\pi$ (this ensures that the halves are well-defined).

In words, if the bounding lines of the two half-planes intersect, then the bounding line of the intermediate half-plane goes through their intersection, bisecting one of the pairs of vertical angles.  If the bounding lines are parallel, then the intermediate half-plane is parallel to both and half-way between them. In either case, the intermediate half-plane contains the intersection of the two half-planes of the admissible pair.
\end{defi}
\begin{prop}\label{prop:venn}
Let $\diskone_0$ and $\diskone_1$ be properly overlapping disks.  Choose a half plane $\halfplane_\iti$ containing $\diskone_\iti$ whose boundary line is tangent to $\diskone_\iti$.  If the pair $(\halfplane_0,\halfplane_1)$ is admissible, then $\diskone_{\frac{1}{2}}$ is contained in the intermediate half-plane of the pair.
\end{prop}
\begin{proof}
Let $\halfplane_\iti=(\angleof_\iti,\displacement_\iti)$.  It suffices to show that
\begin{multline*}\left\langle \frac{\cent_1\radi_0+\cent_0\radi_1}{\radi_0+\radi_1},\left(\cos\left(\frac{\angleof_0+\angleof_1}{2}\right),\sin\left(\frac{\angleof_0+\angleof_1}{2}\right)\right)\right\rangle
\\\ge
\frac{\displacement_0+\displacement_1}{2\cos\left(\frac{\angleof_0-\angleof_1}{2}\right)}+\sqrt{\radi_0\radi_1}\frac{\sqrt{(\radi_0+\radi_1)^2-|\cent_0-\cent_1|^2}}{\radi_0+\radi_1}.
\end{multline*}
By assumption, $\langle \cent_\iti, (\cos\angleof_\iti,\sin\angleof_\iti)\rangle\ge \displacement_\iti+\radi_\iti$.  Using this and the identity \[
\cos\left(\frac{\angleof_0+\angleof_1}{2}\right)=\frac{\cos\angleof_0+\cos\angleof_1}{2\cos\left(\frac{\angleof_0-\angleof_1}{2}\right)},
\]
along with the cognate identity for sine, means that it suffices to show
\begin{multline*}
  \left\langle \cent_1\radi_0+\cent_0\radi_1,\left(\cos\angleof_0+\cos\angleof_1,\sin\angleof_0+\sin\angleof_1\right)\right\rangle
\\\ge
(\radi_0+\radi_1)\left(-\radi_0-\radi_1+\langle \cent_0,(\cos\angleof_0,\sin\angleof_0)\rangle+\langle \cent_1,(\cos\angleof_1,\sin\angleof_1)\rangle\right) \\+2\sqrt{\radi_0\radi_1}\sqrt{(\radi_0+\radi_1)^2-|\cent_0-\cent_1|^2}\cos\left(\frac{\angleof_0-\angleof_1}{2}\right).
\end{multline*}  
Rearranging, this is the same as
\begin{multline*}
(\radi_0+\radi_1)^2 + \radi_1\langle \cent_0-\cent_1,(\cos\angleof_1,\sin\angleof_1)\rangle + \radi_0\langle \cent_1-\cent_0,(\cos\angleof_0,\sin\angleof_0)\rangle
\\\ge
\sqrt{4\radi_0\radi_1}\sqrt{(\radi_0+\radi_1)^2-|\cent_0-\cent_1|^2}\cos\left(\frac{\angleof_0-\angleof_1}{2}\right).
\end{multline*}
If $\cent_0=\cent_1$,  then squaring both sides clearly yields the desired inequality. So assume that $\cent_0\ne \cent_1$, so that we can write $\cent_0-\cent_1=|\cent_0-\cent_1|(\cos\randomangle,\sin\randomangle)$ for some angle $\randomangle$.  Then only the second and third terms have any dependence on $\randomangle$; explicitly they are:
\[|\cent_0-\cent_1|(\radi_1\cos\randomangle\cos\angleof_1+\radi_1\sin\randomangle\sin\angleof_1-\radi_0\cos\randomangle\cos\angleof_0-\radi_0\sin\randomangle\sin\angleof_0)\]
this quantity reaches its maximum and minimum values when
\[
\cos\randomangle=\pm\frac{\radi_0\cos\angleof_0-\radi_1\cos\angleof_1}{\sqrt{\radi_0^2+\radi_1^2-2\radi_0\radi_1\cos(\angleof_0-\angleof_1)}};\
\sin\randomangle=\pm\frac{\radi_0\sin\angleof_0-\radi_1\sin\angleof_1}{\sqrt{\radi_0^2+\radi_1^2-2\radi_0\radi_1\cos(\angleof_0-\angleof_1)}},\]
with the same sign in both cases.  The minimum value is \[-\sqrt{\radi_0^2+\radi_1^2-2\radi_0\radi_1\cos(\angleof_0-\angleof_1)},\] so it suffices to show the inequality
\begin{multline*}
(\radi_0+\radi_1)^2 -\sqrt{\radi_0^2+\radi_1^2-2\radi_0\radi_1\cos(\angleof_0-\angleof_1)}|\cent_0-\cent_1|
\\\ge
\sqrt{4\radi_0\radi_1}\sqrt{(\radi_0+\radi_1)^2-|\cent_0-\cent_1|^2}\cos\left(\frac{\angleof_0-\angleof_1}{2}\right).
\end{multline*}  
The left hand side of this inequality is nonnegative, since each of the factors in the second term is less than or equal to $\radi_0+\radi_1$, so it suffices to show the square of this inequality.  Using $\displaystyle\cos^2\left(\frac{\angleof_0-\angleof_1}{2}\right)=\frac{1+\cos(\angleof_0-\angleof_1)}{2}$, and rearranging, this is:
\[
(\radi_0+\radi_1)^2\left(\sqrt{\radi_0^2+\radi_1^2-2\radi_0\radi_1\cos(\angleof_0-\angleof_1)}-|\cent_0-\cent_1|\right)^2\ge 0,
\]
which is certainly true.
\end{proof}

\begin{cor}\label{cor:halfvenn}
Let $\diskone_0$ and $\diskone_1$ be properly overlapping disks, and let $\disktwo_0$ and $\disktwo_1$ be properly overlapping disks.  Suppose also that the distances between $\diskone_\iti$ and $\disktwo_\iti$ are at least $\smallepsilon_\iti>0$, so that $\diskone_\iti$ and $\disktwo_\iti$ do not overlap. Then the distance from $\diskone_{\frac{1}{2}}$ to $\disktwo_{\frac{1}{2}}$ is at least $\frac{\smallepsilon_0+\smallepsilon_1}{2}$.
\end{cor}
\begin{proof}
Choose the half-plane $\halfplane_{\diskone_\iti}=(\angleof_{\diskone_\iti},\displacement_{\diskone_\iti})$  containg $\diskone_\iti$ whose boundary line is tangent to $\diskone_\iti$ at the closest boundary point to the center of $\disktwo_\iti$, and vice versa.  Then $\angleof_{\diskone_\iti}=-\angleof_{\disktwo_\iti}$ and $\halfplane_{\diskone_\iti}$ and $\halfplane_{\disktwo_\iti}$ have empty intersection.  It is easy to see that if $\angleof_{\diskone_0}=-\angleof_{\diskone_1}$, that either $\halfplane_{\diskone_0}\cap \halfplane_{\diskone_1}$ or $\halfplane_{\disktwo_0}\cap \halfplane_{\disktwo_1}$ is empty.  But $\halfplane_{\diskone_0}\cap \halfplane_{\diskone_1}$ contains $\diskone_0\cap \diskone_1$, which is nonempty by assumption, and likewise for $\disktwo_\iti$.  So the pair $\halfplane_{\diskone_\iti}$ (likewise $\halfplane_{\disktwo_\iti}$) is admissible.

Then the angles of the intermediate halfplanes $\halfplane_\diskone=(\angleof_\diskone,\displacement_\diskone)$ and $\halfplane_\disktwo=(\angleof_\disktwo,\displacement_\disktwo)$ are additive inverses of one another.  The distance between these two half-planes is then $\min\{0,\displacement_\diskone+\displacement_\disktwo\}$.  But $\displaystyle \cos\frac{\angleof_{\diskone_0}-\angleof_{\diskone_1}}{2}=\cos\frac{\angleof_{\disktwo_0}-\angleof_{\disktwo_1}}{2}>0$, which implies:
\[
  \displacement_\diskone+\displacement_\disktwo=\frac{\displacement_{\diskone_0}+\displacement_{\diskone_1}+\displacement_{\disktwo_0}+\displacement_{\disktwo_1}}{2\cos\frac{\angleof_{\diskone_0}-\angleof_{\diskone_1}}{2}}\ge \frac{\displacement_{\diskone_0}+\displacement_{\disktwo_0}}{2}+\frac{\displacement_{\diskone_1}+\displacement_{\disktwo_1}}{2}=\frac{\smallepsilon_0}{2}+\frac{\smallepsilon_1}{2}\] 
\end{proof}
\begin{cor}\label{cor:fullvenn}
Let $\diskone_\iti$, $\disktwo_\iti$ be as in Corollary~\ref{cor:halfvenn}.  Then $\diskone_\timez$ and $\disktwo_\timez$ do not properly overlap one another.
\end{cor}
\begin{proof}
Following the same logic as Corollary~\ref{cor:fullnested}, by induction if $\timez$ is a dyadic rational then the distance between $\diskone_\timez$ and $\disktwo_\timez$ is at least as great as the lesser distance between $\diskone_0$ and $\disktwo_0$ or $\diskone_1$ and $\disktwo_1$.  Then the same is true for all $\timez$ since the dyadic rationals are dense.
\end{proof}
\begin{acknowledgements}
The author would like to thank Kevin Costello, Joseph Hirsh, and John Terilla for many helpful conversations, and the referees for several useful suggestions.
\end{acknowledgements}

\bibliography{references-1}{}
\bibliographystyle{amsalpha}
\end{document}